\documentclass[a4paper]{amsart}

\usepackage[english]{babel}
\usepackage{a4,exscale,amsfonts,amssymb,amsmath,amscd,graphicx,color}
\usepackage{tikz,caption}

\title[The Stationary Stokes Problem in Exterior Domains]
{\sf
The Stationary Stokes Problem in Exterior Domains:\\
Estimates of the Distance to Solenoidal Fields\\
and Functional A Posteriori Error Estimates}
\author{Dirk Pauly}
\address{Fakult\"at f\"ur Mathematik,
Universit\"at Duisburg-Essen, Campus Essen, Germany\newline
\hspace*{4.2mm}\& Faculty of Information Technology, University of Jyv\"askyl\"a, Finland}
\email[Dirk Pauly]{dirk.pauly@uni-due.de}
\author{Sergey Repin}
\address{V.A. Steklov Institute of Mathematics at St. Petersburg, Russia\newline
\hspace*{4.2mm}\& Faculty of Information Technology, University of Jyv\"askyl\"a, Finland}
\email[Sergey Repin]{repin@pdmi.ras.ru}
\date{\today}

\allowdisplaybreaks
\newtheorem{lem}{Lemma}[section]
\newtheorem{theo}[lem]{Theorem}
\newtheorem{cor}[lem]{Corollary}
\newtheorem{rem}[lem]{Remark}

\DeclareMathOperator{\supp}{supp}
\DeclareMathOperator{\dist}{dist}
\newcommand{\reals}{\mathbb{R}}
\newcommand{\I}{\mathbb{I}}
\newcommand{\om}{\Omega}
\newcommand{\ga}{\Gamma}
\newcommand{\calD}{\mathcal{D}}
\newcommand{\calB}{\mathcal{B}}

\newcommand{\sfl}{\mathsf{L}}
\newcommand{\sfc}{\mathsf{C}}
\newcommand{\sfh}{\mathsf{H}}

\newcommand{\sfd}{\mathsf{D}}
\newcommand{\sfs}{\mathsf{S}}
\newcommand{\sfw}{\mathsf{W}}

\newcommand{\lt}{\sfl^{2}}
\newcommand{\li}{\sfl^{\infty}}
\newcommand{\ci}{\sfc^{\infty}}
\newcommand{\ho}{\sfh^{1}}
\newcommand{\hoh}{\sfh^{1/2}}
\newcommand{\woi}{\sfw^{1,\infty}}

\renewcommand{\d}{\sfd}
\newcommand{\s}{\sfs}

\DeclareMathOperator{\esssup}{ess\,sup}
\DeclareMathOperator{\na}{\nabla}
\DeclareMathOperator{\rot}{rot}

\DeclareMathOperator{\divergence}{div}
\renewcommand{\div}{\divergence}
\DeclareMathOperator{\Div}{Div}

\newcommand{\norm}[1]{\|#1\|}
\newcommand{\bnorm}[1]{\big\|#1\big\|}

\newcommand{\scp}[2]{\langle#1,#2\rangle}
\newcommand{\bscp}[2]{\big\langle#1,#2\big\rangle}

\newcommand{\setb}[2]{\big\{#1\,:\,#2\big\}}

\begin{document}

\begin{abstract}

%

This paper is concerned with the analysis of the inf-sup condition
arising in the stationary Stokes problem in exterior domains. 
We deduce values of the
constant in the stability lemma, which yields fully
computable estimates of the distance to the set
of divergence free fields defined in exterior domains.
Using these estimates we obtain computable majorants
of the difference between the exact solution of the Stokes
problem in exterior domains and any approximation from the
admissible (energy) class of functions satisfying 
the Dirichlet boundary condition exactly.
\end{abstract}

\maketitle
\begin{center}
\sf{Dedicated to the 110th anniversary of Solomon Grigor'evich Mikhlin\\
(April 23, 1908 -- August 29, 1990)}
\end{center}
\tableofcontents


\section{Introduction}
\label{secintro}

Let $\omega\subset\reals^{d}$, $d\geq2$, be a bounded domain with Lipschitz boundary $\gamma$,
which is composed of two open and disjoint parts $\gamma_{D},\gamma_{N}\subset\gamma$
(Dirichlet and Neumann part) with $\overline{\gamma}=\overline{\gamma}_{D}\cup\overline{\gamma}_{N}$.
Let the usual Lebesgue and Sobolev spaces (scalar, vector, or tensor valued)
be introduced by $\lt(\omega)$ and $\ho(\omega)$, respectively.
The standard inner product, norm, resp. orthogonality in $\lt(\omega)$ will be denoted by
$\scp{\,\cdot\,}{\,\cdot\,}_{0,\omega}$, $\norm{\,\cdot\,}_{0,\omega}$, resp. $\bot_{0,\omega}$.
For $\gamma_{D}\neq\emptyset$ let $\ho_{\gamma_{D}}(\omega)$ denote the subspace of $\ho(\omega)$ with vanishing
full traces on $\gamma_{D}$.
Moreover, we define spaces with vanishing mean value by\footnote{Throughout the paper
we do not express the respective measure in the notation of integrals, so that, e.g.,
with often used notations
$$\int_{\omega}f=\int_{\omega}fd\lambda=\int_{\omega}fdx,\qquad
\int_{\gamma}f=\int_{\gamma}fdo=\int_{\gamma}fds.$$}
\begin{align*}
\lt_{\bot}(\omega)
&:=\lt(\omega)\cap\reals^{\bot_{0,\omega}}
=\setb{\varphi\in\lt(\omega)}{\int_{\omega}\varphi=0},\\
\ho_{\bot}(\omega)
&:=\ho(\omega)\cap\lt_{\bot}(\omega)
=\setb{\varphi\in\ho(\omega)}{\int_{\omega}\varphi=0}
\end{align*}
and introduce 
\begin{align*}
\lt_{\gamma_{D}}(\omega)
&:=\begin{cases}\lt(\omega)&\text{, if }\gamma_{D}\neq\gamma,\\
\lt_{\bot}(\omega)&\text{, if }\gamma_{D}=\gamma,\end{cases}
&
\ho_{\gamma_{D}}(\omega)
&:=\begin{cases}\ho_{\gamma_{D}}(\omega)&\text{, if }\gamma_{D}\neq\emptyset,\\
\ho_{\bot}(\omega)&\text{, if }\gamma_{D}=\emptyset.\end{cases}
\end{align*}
Furthermore, let us define solenoidal (divergence free) subspaces of $\ho(\om)$ by
$$\s(\omega)
:=\setb{\phi\in\ho(\omega)}{\div\phi=0},\qquad
\s_{\gamma_{D}}(\omega)
:=\ho_{\gamma_{D}}(\omega)\cap\s(\omega).$$
For further notation we refer to Section \ref{secprelim}.
From results of Babuska and Aziz, Ladyzhenskaya and Solonnikov, Brezzi, Necas \cite{BaAz,Br,La,LaSo,Necas},
for mixed boundary conditions see, e.g., the recent results in 
\cite{bauerneffpaulystarkenewpoincare,bauerneffpaulystarkedevDivsymCurl},
we have the following very important lemma 
in the theory of fluid dynamics and other fields of partial differential equations:

\begin{lem}[stability lemma]
\label{stablembddom}
There exists $c>0$ such that for any $g\in\lt_{\gamma_{D}}(\omega)$ 
there is a vector field $u\in\ho_{\gamma_{D}}(\omega)$ with $\div u=g$ and
$\norm{\na u}_{0,\omega}\leq c\norm{g}_{0,\omega}$.
The best constant $c$ will be denoted by $\kappa(\omega,\gamma_{D})$.
\end{lem}

\begin{rem}
\label{stablembddomrem}
Let us note the following:
\begin{itemize}
\item[\bf(i)]
In the theory of electrodynamics
$u$ is called a regular potential as it admits for Maxwell's equations
an unphysically high regularity and a very unphysical boundary condition,
much stronger than the usual normal boundary condition related to the divergence operator.
\item[\bf(ii)]
For $u\in\ho_{\gamma_{D}}(\omega)$
we have the Friedrichs/Poincar\'e inequality 
$\norm{u}_{0,\omega}\leq c\norm{\na u}_{0,\omega}$.
The best constant $c$ is the Friedrichs/Poincar\'e constant
and will be denoted by $c_{fp}(\omega,\gamma_{D})$. 
Hence we conclude for $u$ from Lemma \ref{stablembddom}
$$\frac{1}{c_{fp}(\omega,\gamma_{D})}\norm{u}_{0,\omega}
\leq\norm{\na u}_{0,\omega}
\leq\kappa(\omega,\gamma_{D})\norm{\div u}_{0,\omega}.$$
\item[\bf(iii)]
Note that $\kappa(\omega,\gamma_{D})$ is the norm of the right inverse $g\mapsto u$.
\end{itemize}
\end{rem}

Lemma \ref{stablembddom}  is a keystone fact in the theory of incompressible fluids.
It generates several important corollaries. First of all, Lemma \ref{stablembddom} 
guarantees the solvability of the stationary Stokes problem (in the velocity-pressure posing).
Indeed by solving $g=\div u$ Lemma \ref{stablembddom} yields immediately the following famous 
inf-sup of LBB condition:

\begin{cor}[inf-sup lemma]
\label{infsuplembddom}
It holds
$$\inf_{g\in\lt_{\gamma_{D}}(\omega)}\sup_{u\in\ho_{\gamma_{D}}(\omega)}
\frac{\scp{g}{\div u}_{0,\omega}}{\norm{g}_{0,\omega}\norm{\na u}_{0,\omega}}\geq\frac{1}{\kappa(\omega,\gamma_{D})}.$$
\end{cor}

A solution theory for the Stokes problem follows. 
The stationary Stokes problem reads as follows:
For given $\nu>0$, $G\in\lt(\omega)$, $u_{D}\in\s(\omega)$, $\sigma_{N}$
find a velocity field $u$ and a pressure function $p$ solving
the first order system
\begin{center}
\begin{minipage}{60mm}
\begin{align*}
-\Div\sigma
&=G
&
&\text{in }\omega,\\
\sigma
&=\nu\na u-p\,\I
&
&\text{in }\omega,\\
-\div u
&=0
&
&\text{in }\omega,\\
u
&=u_{D}
&
&\text{on }\gamma_{D},\\
\sigma n
&=\sigma_{N}
&
&\text{on }\gamma_{N}.
\end{align*}
\end{minipage}
\end{center}
Equivalently, by removing the additional stress tensor $\sigma$,
we have the second order formulation
\begin{center}
\begin{minipage}{60mm}
\begin{align*}
-\nu\Delta u+\na p
&=G
&
&\text{in }\omega,\\
-\div u
&=0
&
&\text{in }\omega,\\
u
&=u_{D}
&
&\text{on }\gamma_{D},\\
(\nu\na u-p)n
&=\sigma_{N}
&
&\text{on }\gamma_{N}.
\end{align*}
\end{minipage}
\end{center}
It is worth noting that the Dirichlet boundary term $u_{D}$ satisfies
\begin{align}
\label{meanbcbddom}
\int_{\gamma}n\cdot u_{D}
=\int_{\omega}\div u_{D}=0.
\end{align}
Hence, if the boundary datum is given by some $\tilde{u}_{D}\in\hoh(\gamma)$
any solenoidal extension $u_{D}$ to $\omega$ of $\tilde{u}_{D}$
must satisfy the normal mean value property
$$\int_{\gamma}n\cdot\tilde{u}_{D}=0.$$
On the other hand, one can always find a continuous and solenoidal lifting 
of a boundary term $\tilde{u}_{D}\in\hoh(\gamma)$ as long as \eqref{meanbcbddom} holds,
see also our more general Corollary \ref{inhomodistlembddom}.
In the smooth case we have for $\phi\in\ci_{\gamma_{D}}(\omega)$
\begin{align*}
&\qquad-\nu\scp{\Delta u}{\phi}_{0,\omega}
=\nu\scp{\na u}{\na\phi}_{0,\omega}
-\nu\bscp{(\na u)n}{\phi}_{0,\gamma_{N}}\\
&=\scp{G}{\phi}_{0,\omega}
-\scp{\na p}{\phi}_{0,\omega}
=\scp{G}{\phi}_{0,\omega}
+\scp{p}{\div\phi}_{0,\omega}
-\bscp{pn}{\phi}_{0,\gamma_{N}},
\end{align*}
i.e.,
$$\nu\scp{\na u}{\na\phi}_{0,\omega}
=\scp{G}{\phi}_{0,\omega}
+\scp{p}{\div\phi}_{0,\omega}
+\scp{\sigma_{N}}{\phi}_{0,\gamma_{N}}.$$
Let us for simplicity assume $\sigma_{N}=0$.
A possible variational formulation is given by the following:
Find $u\in u_{D}+\s_{\gamma_{D}}(\omega)$ such that for all $\phi\in\s_{\gamma_{D}}(\omega)$
$$\nu\scp{\na u}{\na\phi}_{0,\omega}
=\scp{G}{\phi}_{0,\omega}.$$
Using the ansatz $u=u_{D}+\hat{u}$ with $\hat{u}\in\s_{\gamma_{D}}(\omega)$ 
we reduce this formulation to find $\hat{u}\in\s_{\gamma_{D}}(\omega)$ 
such that for all $\phi\in\s_{\gamma_{D}}(\omega)$
$$\nu\scp{\na\hat{u}}{\na\phi}_{0,\omega}
=\scp{G}{\phi}_{0,\omega}
-\nu\scp{\na u_{D}}{\na\phi}_{0,\omega}.$$
Note that the pressure $p$ is not involved in this formulation.
Another formulation taking the pressure into account 
and removing the unpleasant solenoidal condition from the Hilbert space
is the following saddle point formulation: 
Find $(\hat{u},p)\in\ho_{\gamma_{D}}(\omega)\times\lt_{\gamma_{D}}(\omega)$
such that for all $(\phi,\varphi)\in\ho_{\gamma_{D}}(\omega)\times\lt_{\gamma_{D}}(\omega)$
\begin{align*}
\nu\scp{\na\hat{u}}{\na\phi}_{0,\omega}
-\scp{p}{\div\phi}_{0,\omega}
&=\scp{G}{\phi}_{0,\omega}
-\nu\scp{\na u_{D}}{\na\phi}_{0,\omega},\\
-\scp{\div\hat{u}}{\varphi}_{0,\omega}
&=0,
\end{align*}
which reads in formal matrix notation (boundary conditions are indicated as subscripts) as
$$\begin{bmatrix}
-\nu\div_{\gamma_{N}}\na_{\gamma_{D}} & \na_{\gamma_{N}} \\
-\div_{\gamma_{D}} & 0
\end{bmatrix}
\begin{bmatrix}
\hat{u} \\
p
\end{bmatrix}
=
\begin{bmatrix}
G+\nu\div_{\gamma_{N}}\na_{\gamma_{D}}u_{D} \\
0
\end{bmatrix}.$$

\begin{cor}[Stokes lemma]
\label{stokeslembddom}
For $\nu>0$, $G\in\lt_{\gamma_{N}}(\omega)$, 
$u_{D}\in\s(\omega)$ the Stokes system is uniquely solvable with 
$u=u_{D}+\hat{u}\in u_{D}+\s_{\gamma_{D}}(\omega)\subset\s(\omega)$
and $p\in\lt_{\gamma_{D}}(\omega)$. Moreover, 
\begin{align*}
\nu\norm{\na\hat{u}}_{0,\omega}
&\leq c_{fp}(\omega,\gamma_{D})\norm{G}_{0,\omega}
+\nu\norm{\na u_{D}}_{0,\omega},\\
\nu\norm{\na u}_{0,\omega}
&\leq c_{fp}(\omega,\gamma_{D})\norm{G}_{0,\omega}
+2\nu\norm{\na u_{D}}_{0,\omega},\\
\norm{p}_{0,\omega}
&\leq 2\kappa(\omega,\gamma_{D})\big(c_{fp}(\omega,\gamma_{D})\norm{G}_{0,\omega}
+\nu\norm{\na u_{D}}_{0,\omega}\big).
\end{align*}
\end{cor}

Note that here in the vector valued case
\begin{align*}
\lt_{\gamma_{N}}(\omega)
&=\begin{cases}\lt(\omega)&\text{, if }\gamma_{D}\neq\emptyset,\\
\lt_{\bot}(\omega)&\text{, if }\gamma_{D}=\emptyset,\end{cases}\\
\lt_{\bot}(\omega)
&=\lt(\omega)\cap(\reals^{d})^{\bot_{0,\omega}}
=\setb{\phi\in\lt(\omega)}{\int_{\omega}\phi_{i}=0}.
\end{align*}

\begin{proof}
Standard saddle point theory and the inf-sup lemma, Corollary \ref{infsuplembddom},
shows existence and the estimates follow by standard arguments,
which provide also uniqueness. Note that we solve $p=\div\phi$ by Lemma \ref{stablembddom}
to get the estimates for the pressure $p$.
\end{proof}

Another direct consequence of Lemma \ref{stablembddom} is an estimate for the distance of vector fields
to solenoidal fields, more precisely:

\begin{cor}[distance lemma]
\label{distlembddom}
For any $u\in\ho_{\gamma_{D}}(\omega)$ there exists 
a solenoidal $u_{0}\in\s_{\gamma_{D}}(\omega)$ such that
$$\dist\big(u,\s_{\gamma_{D}}(\omega)\big)
=\inf_{\phi\in\s_{\gamma_{D}}(\omega)}\bnorm{\na(u-\phi)}_{0,\omega}
\leq\bnorm{\na(u-u_{0})}_{0,\omega}
\leq\kappa(\omega,\gamma_{D})\norm{\div u}_{0,\omega}.$$
\end{cor}

\begin{proof}
For $u\in\ho_{\gamma_{D}}(\omega)$ solve 
$\div\tilde{u}=\div u\in\lt_{\gamma_{D}}(\omega)$ with $\tilde{u}\in\ho_{\gamma_{D}}(\omega)$
and the stability estimate $\norm{\na\tilde{u}}_{0,\omega}\leq\kappa(\omega,\gamma_{D})\norm{\div u}_{0,\omega}$
by Lemma \ref{stablembddom}. Note that for $\gamma_{D}=\gamma$ it holds
$$\int_{\omega}\div u
=\int_{\gamma}n\cdot u=0.$$
Then $u_{0}:=u-\tilde{u}\in\s_{\gamma_{D}}(\omega)$ and
$\bnorm{\na(u-u_{0})}_{0,\omega}
=\bnorm{\na\tilde{u}}_{0,\omega}
\leq\kappa(\omega,\gamma_{D})\norm{\div u}_{0,\omega}$.
\end{proof}

This result can be extended to vector fields satisfying non-homogeneous Dirichlet boundary conditions
provided that such a vector field $u$ satisfies 
$\div u\in\lt_{\gamma_{D}}(\om)$, the mean value condition, i.e.,
\begin{align}
\label{meancondbddom}
\int_{\omega}\div u=0\quad\text{, if }\gamma_{D}=\gamma.
\end{align}

\begin{cor}[inhomogeneous distance lemma]
\label{inhomodistlembddom}
For any $u\in\ho(\omega)$ with $\div u\in\lt_{\gamma_{D}}(\omega)$ there exists 
a solenoidal $u_{0}\in\s(\omega)$ such that
$u_{0}-u\in\ho_{\gamma_{D}}(\omega)$, i.e., $u_{0}|_{\gamma_{D}}=u|_{\gamma_{D}}$, and
$$\bnorm{\na(u_{0}-u)}_{0,\omega}
\leq\kappa(\omega,\gamma_{D})\norm{\div u}_{0,\omega}.$$
\end{cor}

Similar estimates for vector fields defined in $\mathsf{W}^{1,q}(\om)$ spaces for $q\in(1,\infty)$ 
have been obtained in \cite{Re2014,Re2015}.
In the literature, results like Corollary \ref{inhomodistlembddom}
are often called lifting lemmas, since a boundary datum
$u|_{\gamma_{D}}$ is lifted to the domain $\omega$,
in this case with a solenoidal representative. Note that
$$\int_{\gamma}n\cdot u=\int_{\omega}\div u=0\quad\text{, if }\gamma_{D}=\gamma.$$

\begin{proof}
For $u\in\ho(\omega)$ solve by Lemma \ref{stablembddom}
$\div\tilde{u}=\div u\in\lt_{\gamma_{D}}(\omega)$ with $\tilde{u}\in\ho_{\gamma_{D}}(\omega)$
and $\norm{\na\tilde{u}}_{0,\omega}\leq\kappa(\omega,\gamma_{D})\norm{\div u}_{0,\omega}$. 
Note that \eqref{meancondbddom} holds for $\gamma_{D}=\gamma$.
Then $u_{0}:=u-\tilde{u}\in\s(\omega)$ with $u-u_{0}=\tilde{u}\in\ho_{\gamma_{D}}(\omega)$ and
$\bnorm{\na(u_{0}-u)}_{0,\omega}
=\bnorm{\na\tilde{u}}_{0,\omega}
\leq\kappa(\omega,\gamma_{D})\norm{\div u}_{0,\omega}$.
\end{proof}

Estimates of the distance to $\s_{\gamma_{D}}(\omega)$ have not only theoretical meaning. 
They are also important for the a posteriori analysis of numerical solutions 
which usually satisfy the divergence free condition only approximately. 
If the constant $\kappa(\omega,\gamma_{D})$ is known, then by using Corollary \ref{distlembddom} 
we can deduce guaranteed and fully computable error bounds for approximations of problems arising 
in the theory of viscous incompressible fluids. 
For problems in bounded Lipschitz domains the respective results 
are presented in \cite{ReAlgAnal2004,ReGruyter}.

In this contribution we extend Lemma \ref{stablembddom}
and its corollaries to the case of exterior domains $\om\subset\reals^{d}$
and investigate applications to estimate the distance of vector fields
to solenoidal fields. These estimates allows us to deduce
new functional a posteriori error estimates valid for a wide class of
approximate solutions to the stationary Stokes problem in exterior domains.

\section{Preliminaries}
\label{secprelim}

Let $\calD\subset\reals^{d}$, $d\geq2$, be a domain (an open and connected set) 
with Lipschitz boundary $\calB$, which is composed of two open and disjoint parts $\calB_{D},\calB_{N}\subset\calB$
with $\overline{\calB}=\overline{\calB}_{D}\cup\overline{\calB}_{N}$ (Dirichlet and Neumann part).
We note that $\calD$ can be bounded or unbounded, especially an exterior domain (a domain with compact complement).
We introduce the usual Lebesgue and Sobolev spaces of square integrable functions
or vector/tensor fields by $\lt(\calD)$ and $\ho(\calD)$, respectively.
The standard inner product, norm, resp. orthogonality in $\lt(\calD)$ are denoted by
$\scp{\,\cdot\,}{\,\cdot\,}_{0,\calD}$, $\norm{\,\cdot\,}_{0,\calD}$, resp. $\bot_{0,\calD}$.
Moreover, let
$$\lt_{\calB_{D}}(\calD):=
\begin{cases}
\lt(\calD)\cap\reals^{\bot_{0,\calD}}&\text{, if }\calB_{D}=\calB,\\
\lt(\calD)&\text{, else,}
\end{cases}$$
provided that $\calD$ is bounded.
If $\calB_{D}\not=\emptyset$, homogeneous
Dirichlet boundary conditions are encoded in 
$\ho_{\calB_{D}}(\calD)$,
defined as closure of 
$$\ci_{\calB_{D}}(\calD):=\setb{\phi|_{\calD}}{u\in\ci(\reals^{d}),\;\supp\phi\text{ compact, }\dist(\supp\phi,\calB_{D})>0}$$
in $\ho(\calD)$.
Moreover, we introduce the polynomially weighted spaces
\begin{align*}
\lt_{\pm1}(\calD)
&:=\setb{\phi\in\lt_{\mathsf{loc}}(\calD)}{\rho^{\pm1}\phi\in\lt(\calD)},\\
\ho_{-1}(\calD)
&:=\setb{\phi\in\lt_{-1}(\calD)}{\na\phi\in\lt(\calD)},
\end{align*}
where the weight function $\rho$ is defined by $\rho(r):=(1+r^2)^{1/2}$, $r(x):=|x|$.
Inner product, norm, resp. orthogonality in $\lt_{\pm1}(\calD)$ is defined and denoted by
$\scp{\,\cdot\,}{\,\cdot\,}_{\pm1,\calD}:=\bscp{\rho^{\pm2}\,\cdot\,}{\,\cdot\,}_{0,\calD}$, 
$\norm{\,\cdot\,}_{\pm1,\calD}$, resp. $\bot_{\pm1,\calD}$.
As before, if $\calB_{D}\not=\emptyset$,
homogeneous (full, tangential, resp. normal) Dirichlet boundary conditions 
are introduced in $\ho_{-1,\calB_{D}}(\calD)$,
the closure of $\ci_{\calB_{D}}(\calD)$ in $\ho_{-1}(\calD)$.
Finally, in particular for the Stokes equations, we introduce spaces of solenoidal fields
\begin{align*}
\s(\calD)
&:=\setb{\phi\in\ho(\calD)}{\div\phi=0},
&
\s_{\calB_{D}}(\calD)
&:=\ho_{\calB_{D}}(\calD)\cap\s(\calD),\\
\s_{-1}(\calD)
&:=\setb{\phi\in\ho_{-1}(\calD)}{\div\phi=0},
&
\s_{-1,\calB_{D}}(\calD)
&:=\ho_{-1,\calB_{D}}(\calD)\cap\s_{-1}(\calD).
\end{align*}
Note that in the case of a bounded domain, there is no difference between the unweighted and weighted spaces,
meaning that the spaces coincide as sets and possess different inner products.

Throughout the paper we assume that $\om\subset\reals^{d}$, 
where $d\geq3$ (the special case $d=2$ is considered 
in Section \ref{2Dsec} and in Appendix II),
is an exterior domain with a Lipschitz boundary $\ga$,
which is composed of two open and disjoint parts $\ga_{D},\ga_{N}\subset\ga$
(Dirichlet and Neumann part) with $\overline{\ga}=\overline{\ga}_{D}\cup\overline{\ga}_{N}$.
Moreover, let $\reals^{d}\setminus\om\subset B_{r_{1}}$ for some $r_{2}>r_{1}>0$ and 
$$\omega:=\om_{r_{2}}:=\om\cap B_{r_{2}},\quad 
\gamma=\ga\cup S_{r_{2}},\quad 
\gamma_{D}:=\ga_{D}\cup S_{r_{2}},$$
where $B_{r}$ and $S_{r}$ denote the open ball and the sphere of radius $r$
centered at the origin in $\reals^{d}$, respectively.
We also pick some cut-off Lipschitz continuous function $\xi\in\woi(\reals;[0,1])$
satisfying $\xi|_{(-\infty,0]}=0$ and $\xi|_{[1,\infty)}=1$ and set
$$\xi'_{\infty}:=\esssup_{[0,1]}|\xi'|.$$
Then the function $\tilde{\xi}$ defined by $\tilde{\xi}(z):=\xi\big((z-r_{1})/(r_{2}-r_{1})\big)$ 
belongs to $\woi(\reals;[0,1])$ as well and satisfies 
$\tilde{\xi}|_{(-\infty,r_{1}]}=0$ and $\tilde{\xi}|_{[r_{2},\infty)}=1$.
Thus 
\begin{align}
\label{defeta}
\eta:=\tilde{\xi}\circ r\in\woi(\reals^{d})
\end{align}
with $\eta|_{\overline{B}_{r_{1}}}=0$ and $\eta|_{\reals^{d}\setminus B_{r_{2}}}=1$.
Finally, we define the constant
$$c_{d}:=
\frac{2}{d-2}.$$

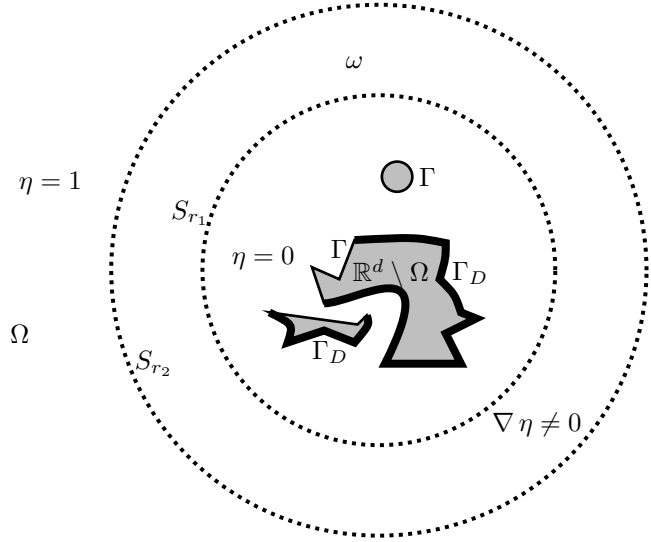
\begin{figure}
\def\scalepic{0.4}
\begin{center}
\begin{tikzpicture}[scale=\scalepic]
\fill [gray,opacity=.5] (0,0) .. controls (1,0) and (4,2) .. (2,-2) -- (4.5,-2) -- (4,-1) -- (5,-0.5) -- (4.5,-0.3) .. controls (4.3,0.3) ..(3.8,0.8) .. controls (4,1.5) .. (4,2) .. controls (3,2.2) .. (1,2.1) -- (0.5,0.8) -- (-0.4,1.2) -- cycle;
\draw [line width=1pt] (0,0) .. controls (1,0) and (4,2) .. (2,-2) -- (4.5,-2) -- (4,-1) -- (5,-0.5) -- (4.5,-0.3) .. controls (4.3,0.3) ..(3.8,0.8) .. controls (4,1.5) .. (4,2) .. controls (3,2.2) .. (1,2.1) -- (0.5,0.8) -- (-0.4,1.2) -- cycle;
\draw [line width=3pt] (0,0) .. controls (1,0) and (4,2) .. (2,-2) -- (4.5,-2) -- (4,-1) -- (5,-0.5) -- (4.5,-0.3) .. controls (4.3,0.3) ..(3.8,0.8) .. controls (4,1.5) .. (4,2) .. controls (3,2.2) .. (1,2.1);
\fill [lightgray] (-1.8,-0.3) .. controls (-1.0,-0.7) .. (-1.2,-1.3) -- (0,-0.9) -- (1.0,-1.3) .. controls (1.5,-0.7) .. (1.4,-0.4) -- (1.1,-0.7) -- cycle;
\draw [line width=1pt] (-1.8,-0.3) .. controls (-1.0,-0.7) .. (-1.2,-1.3) -- (0,-0.9) -- (1.0,-1.3) .. controls (1.5,-0.7) .. (1.4,-0.4) -- (1.1,-0.7) -- cycle;
\draw [line width=3pt] (-1.8,-0.3) .. controls (-1.0,-0.7) .. (-1.2,-1.3) -- (0,-0.9) -- (1.0,-1.3) .. controls (1.5,-0.7) .. (1.4,-0.4);
\fill [lightgray] (2.4,4.2) circle (0.5);
\draw [line width=1pt] (2.4,4.2) circle (0.5);
\draw [black,dotted,line width=1.5pt] (1.8,1.1) circle (8.8);
\draw [black,dotted,line width=1.5pt] (1.8,1.1) circle (5.8);
\node at (1,8) {$\omega$};
\node at (-10,-1) {$\om$};
\node at (-9,4) {$\eta=1$};
\node at (-2,1.5) {$\eta=0$};
\node at (7,-4) {$\na\eta\not=0$};
\node at (2.2,1) {$\reals^{d}\setminus\om$};
\node at (3.4,4.2) {$\ga$};
\node at (0.5,1.8) {$\ga$};
\node at (4.8,1) {$\ga_{D}$};
\node at (0.2,-1.6) {$\ga_{D}$};
\node at (-5.6,-2) {$S_{r_{2}}$};
\node at (-4.4,3) {$S_{r_{1}}$};
\end{tikzpicture}
\end{center}
\caption{$\reals^{d}\setminus\overline{\om}$ (gray)
surrounded by the boundary $\ga$ (thin black lines),
the boundary part $\ga_{D}$ (thick black lines),
and the artificial boundary spheres (dashed lines)}
\label{fig}
\end{figure}

The two main ingredients for our proofs are Lemma \ref{stablembddom}
and a few results from the theory of $\rot$-$\div$-systems in exterior domains, 
which can be summarised in the two subsequent lemmas as follows:

\begin{lem}[Friedrichs/Poincar\'e lemma for exterior domains]
\label{fplemextdom}
The following weighted Friedrichs/Poincar\'e estimates hold:
\begin{itemize}
\item[\bf(i)]
There exists $c>0$ such that for all $v\in\ho_{-1,\ga_{D}}(\om)$ it holds
$$\norm{v}_{-1,\om}\leq c\norm{\na v}_{0,\om}.$$
The best constant $c$ is the Friedrichs/Poincar\'e constant
and is denoted by $c_{fp}(\om,\ga_{D})$.
\item[\bf(ii)]
If $\ga_{D}=\ga$, then $c_{fp}(\om,\ga)$ is the Friedrichs constant $c_{f}(\om)$ and can be estimates by
$$c_{fp}(\om,\ga)=c_{f}(\om)\leq c_{d}.$$
Especially, for all $v\in\ho_{-1,\ga}(\om)$ it holds
$\norm{v}_{-1,\om}
\leq c_{d}\norm{\na v}_{0,\om}$.
\item[\bf(ii')]
If $\ga_{D}=\emptyset$, then $c_{fp}(\om,\emptyset)$ is the Poincar\'e constant $c_{p}(\om)$.
Particularly, for all $v\in\ho_{-1}(\om)$ it holds
$\norm{v}_{-1,\om}
\leq c_{p}(\om)\norm{\na v}_{0,\om}$.
\item[\bf(iii)]
If $\om=\reals^{d}$, then the Friedrichs and Poincar\'e constants 
coincide and, moreover,
$$c_{fp}(\reals^{d})=c_{f}(\reals^{d})=c_{p}(\reals^{d})\leq c_{d}.$$
Especially, for all $v\in\ho_{-1}(\reals^{d})$ it holds
$$\norm{v}_{-1,\reals^{d}}
\leq c_{d}\norm{\na v}_{0,\reals^{d}}.$$
\end{itemize}
\end{lem}

Note that no boundary or mean value conditions are needed
in Lemma \ref{fplemextdom}, since the constants 
are not integrable\footnote{More precisely, 
it holds $(1+r^2)^{t/2}\in\lt_{-1}(\om)$,
if and only if $t-1<-d/2$. Putting $t=0$ shows the assertion.} 
in $\lt_{-1}(\om)$ for $d\geq3$. 
For $d=3$, Lemma \ref{fplemextdom} follows immediately
from \cite[Poincar\'e's estimate IV, p. 62]{leisbook} by approximation.
Nevertheless, we present a simple and self-contained proof.

\begin{proof}
From \cite[Appendix 4.2, Lemma 4.1, Corollary 4.2, Remark 4.3]{paulyrepinell},
see also \cite[Poincar\'e's estimate III, p. 57]{leisbook} and
\cite[Lemma 4.1]{saranenwitschexteriorell}, 
we have for all\footnote{Note that $r^{-1}\in\lt(B_{1})$ if and only if $d\geq3$.}
$u\in\ci_{\ga}(\om)$ 
$$\norm{u}_{-1,\om}
\leq\norm{r^{-1}u}_{0,\om}
\leq c_{d}\norm{\na u}_{0,\om}$$
and hence by density and continuity for all $u\in\ho_{-1,\ga}(\om)$
\begin{align}
\label{fcorextdom}
\norm{u}_{-1,\om}
\leq c_{d}\norm{\na u}_{0,\om}.
\end{align}
For all $u\in\ho_{-1}(\om)$ we see
$\eta u\in\ho_{-1,\ga}(\om)$ and 
$\norm{\eta u}_{-1,\om}
\leq c_{d}\norm{\na(\eta u)}_{0,\om}$
by \eqref{fcorextdom}. Hence
\begin{align}
\label{flemcptpertextdom}
\norm{u}_{-1,\om}
&\leq c_{d}\norm{\na u}_{0,\om}
+c_{d}\norm{u\na\eta}_{0,\om}+\norm{(1-\eta)u}_{-1,\om}
\leq c_{d}\norm{\na u}_{0,\om}
+\tilde{c}_{d}\norm{u}_{0,\omega},
\end{align}
where\footnote{For $r_{2}=r_{1}+1$ and $\xi'_{\infty}\leq1$ we have $\tilde{c}_{d}\leq c_{d}+1$.}
$\tilde{c}_{d}:=c_{d}\xi'_{\infty}/(r_{2}-r_{1})+1$.
Now we can prove (i), even the stronger result (ii').
If the estimate in (ii') is false, there is a sequence $(u_{n})\subset\ho_{-1}(\om)$
with $\norm{u_{n}}_{-1,\om}=1$ and $\norm{\na u_{n}}_{0,\om}<1/n$. 
Hence, $(u_{n})$ is bounded in $\ho(\omega)$ as well. By Rellich's selection theorem we can assume w.l.o.g.
that $(u_{n})$ already converges in $\lt(\omega)$. Thus, by \eqref{flemcptpertextdom} $(u_{n})$ is a Cauchy sequence 
in $\lt_{-1}(\om)$ and hence also in $\ho_{-1}(\om)$. Therefore, $(u_{n})$ converges in $\ho_{-1}(\om)$
to some $u\in\ho_{-1}(\om)$ with $\na u=0$. We conclude that $u$ is constant. 
But then $u\in\lt_{-1}(\om)$ must vanish, which implies a contradiction by
$1=\norm{u_{n}}_{-1,\om}\to0$.
\end{proof}

\begin{lem}[$\rot$-$\div$ lemma for the whole space]
\label{maxlemextdom}
For any $f\in\lt(\reals^{d})$ there exists a unique
$v\in\ho_{-1}(\reals^{d})$ such that $\rot v=0$, $\div v=f$, and
$$\frac{1}{c_{d}}\norm{v}_{-1,\reals^{d}}
\leq\norm{\na v}_{0,\reals^{d}}
=\norm{\div v}_{0,\reals^{d}}
=\norm{f}_{0,\reals^{d}}.$$
\end{lem}

Note that the equation $-\Delta=\rot^{*}\rot-\na\div$ implies
\begin{align}
\label{narotdiv}
\norm{\na\Phi}_{0,\reals^{d}}^2
=\norm{\rot\Phi}_{0,\reals^{d}}^2
+\norm{\div\Phi}_{0,\reals^{d}}^2
\end{align}
for all $\Phi\in\ci(\reals^{d})$ having compact support
and extends to all $\Phi\in\ho_{-1}(\reals^{d})$ by density and continuity.
Hence the equality $\norm{\na v}_{0,\reals^{d}}=\norm{\div v}_{0,\reals^{d}}$ 
in Lemma \ref{maxlemextdom} follows immediately.
The results of Lemma \ref{maxlemextdom} are well known and can be found, e.g.,
in \cite{picardharmdiff,picardpotential,picardboundaryelectro} 
or in \cite{kuhnpaulyregmax,paulystatic}. 
In particular, Lemma \ref{maxlemextdom} follows from 
Lemma \ref{fplemextdom} (iii), \eqref{narotdiv}, and
\cite[Theorem A.7, Theorem 3.2 (ii)]{kuhnpaulyregmax}, 
see also \cite[Lemma 3.5, Lemma 3.6, Theorem 4.1]{paulystatic}. 

\section{The Stability Lemma for Exterior Domains}
\label{seclbbextdom}

First we define our upper bound related to the geometry presented in Figure \ref{fig}.
\begin{align}
\label{defkappaupperbound}
\hat{\kappa}(\om,\ga_{D})
&:=(1+\kappa)\big(1+c_{d}\frac{\xi'_{\infty}\rho(r_{2})}{r_{2}-r_{1}}\big),
&
\kappa
&:=\min\big\{\kappa(\omega,\gamma_{D}),\kappa(\omega,\gamma)\big\}.
\end{align}
Especially for $r_{2}=r_{1}+1$ and $\xi'_{\infty}\leq1$ we have
the simple upper bound
\begin{align*}
\hat{\kappa}(\om,\ga_{D})
&=(1+\kappa)\big(1+c_{d}\rho(r_{2})\big).
\end{align*}
The above constants contain the stability constants
$\kappa(\omega,\gamma_{D})$, $\kappa(\omega,\gamma)$
associated with the bounded domain $\omega$ 
and respective parts of its boundary $\gamma_{D}$ and $\gamma$.

\begin{rem}
\label{constrem}
$\kappa(\omega,\gamma_{D})$ and $\kappa(\omega,\ga_{D})$ 
depend on $r_{2}$, so that the best value of $r_{2}$
(which minimises the constant) is not known a priori
and has to be optimized by some algebraic procedure.
We emphasize that 
$$\kappa\leq\kappa(\omega,\gamma_{D}),\qquad
\kappa\leq\kappa(\omega,\gamma).$$
and that a bound in the simple situation from above is given by
$$\hat{\kappa}(\om,\ga_{D})
=(1+\kappa)\big(1+\frac{2\sqrt{2}}{d-2}r_{2}\big)$$
\end{rem}

Now we can proceed to prove the stability lemma for exterior domains.
First we observe a trivial case for compactly supported right hand sides:

\begin{rem}
\label{stablemextdomrem}
There exists $c>0$ such that
for all $f\in\lt(\om)$ with $\supp f\subset\overline{\omega}$
and in the case $\ga_{D}=\ga$, additionally, $\int_{\om}f=0$, 
there is a vector field $v\in\ho_{-1,\ga_D}(\om)$ with
$\div v=f$ and $\norm{\na v}_{0,\om}\leq c\norm{f}_{0,\om}$.
The best constant is denoted by $\kappa(\om,\ga_{D})$. 
Moreover, $v$ can be chosen with compact support in $\overline{\omega}$, 
in particular, $v\in\ho_{\gamma_D}(\omega)\subset\ho_{\ga_D}(\om)$.
In this case, $\kappa(\om,\ga_{D})=\kappa(\omega,\gamma_{D})$.
For a short proof, we set $g:=f|_{\omega}\in\lt_{\gamma_{D}}(\omega)$ and 
by Lemma \ref{stablembddom} there exist $\kappa(\omega,\gamma_{D})>0$ and $u\in\ho_{\gamma_{D}}(\omega)$ with 
$$\div u=g,\quad
\norm{\na u}_{0,\omega}\leq\kappa(\omega,\gamma_{D})\norm{g}_{0,\omega}.$$
Then $v$, which is the extension by zero of $u$ to $\om$,
belongs to $\ho_{\ga_{D}}(\om)$ and $\supp v=\supp u\subset\overline{\omega}$.
Moreover, $\div v=f$ and 
$\norm{\na v}_{0,\om}
=\norm{\na u}_{0,\omega}
\leq\kappa(\omega,\gamma_{D})\norm{g}_{0,\omega}
=\kappa(\omega,\gamma_{D})\norm{f}_{0,\om}$.
\end{rem}

Our main result reads as follows:

\begin{lem}[stability lemma for exterior domains]
\label{stablemextdom}
There exists $c>0$ such that for all $f\in\lt(\om)$ 
there is a vector field $v\in\ho_{-1,\ga_D}(\om)$ with
$$\div v=f\quad\text{and}\quad\norm{\na v}_{0,\om}\leq c\norm{f}_{0,\om}.$$
The best constant is denoted by $\kappa(\om,\ga_{D})$
which equals the norm of the corresponding right inverse $f\mapsto v$. 
Moreover with \eqref{defkappaupperbound}
$$\kappa(\om,\ga_{D})\leq\hat{\kappa}(\om,\ga_{D}).$$
\end{lem}

Note that no mean value condition is imposed on $f$.

\begin{proof}
We extend $f$ by $0$ to $\reals^{d}\setminus\overline{\om}$ and identify $f\in\lt(\reals^{d})$. 
By Lemma \ref{maxlemextdom} we get some
$v\in\ho_{-1}(\reals^{d})$ with $\rot v=0$ solving $\div v=f$ in $\reals^{d}$ and
\begin{align}
\label{navf}
\norm{v}_{-1,\reals^{d}}
\leq c_{d}\norm{\na v}_{0,\reals^{d}},\qquad
\norm{\na v}_{0,\reals^{d}}=\norm{\div v}_{0,\reals^{d}}=\norm{f}_{0,\om}.
\end{align}
Then $\eta v\in\ho_{-1,\ga}(\om)$ with \eqref{defeta}
and $\supp(\eta v)\subset\reals^{d}\setminus B_{r_{1}}$.
We are searching for $v\in\ho_{-1,\ga_D}(\om)$ solving $\div v=f$ in the form
$$v:=\eta v+v_{\omega},$$
where $v_{\omega}\in\ho_{\ga_D}(\om)$ with $\supp v_{\omega}\subset\overline{\omega}$ 
is the extension by zero to $\om$ of some vector field $u\in\ho_{\gamma_{D}}(\omega)$.
Hence, $v$ and $v_{\omega}$ should satisfy
$$f=\div v=\eta f+\na\eta\cdot v+\div v_{\omega}\quad\text{in }\om$$
and we have to find $u\in\ho_{\gamma_{D}}(\omega)$ with
$$\div u=g
:=(1-\eta)f-\na\eta\cdot v\in\lt(\omega)\quad\text{in }\omega.$$
Note that indeed $\supp(1-\eta)\subset\overline{B}_{r_{2}}$,
$\supp\na\eta\subset\overline{B}_{r_{2}}\setminus B_{r_{1}}$
and hence $\supp g\subset\overline{\omega}$. Moreover,
$$g=(1-\eta)f+\na(1-\eta)\cdot v
=\div\big((1-\eta)v\big)\quad\text{in }\reals^{d}$$
and therefore
$$\int_{\omega}g
=\int_{\gamma}(1-\eta)n\cdot v
=\int_{\ga}n\cdot v
=-\int_{\reals^{d}\setminus\overline{\om}}\div v
=-\int_{\reals^{d}\setminus\overline{\om}}f
=0.$$
Thus, $g$ has mean value zero independent of the particular boundary condition on $\ga_{D}$,
i.e., $g\in\lt_{\gamma_{D}}(\omega)$.
Lemma \ref{stablembddom} provides such a $u\in\ho_{\gamma_{D}}(\omega)$ with
$\norm{\na u}_{0,\omega}\leq\kappa(\omega,\gamma_{D})\norm{g}_{0,\omega}$.
We can even pick $u\in\ho_{\gamma}(\omega)\subset\ho_{\gamma_{D}}(\omega)$ with
$\norm{\na u}_{0,\omega}\leq\kappa(\omega,\gamma)\norm{g}_{0,\omega}$.
Hence, generally, we obtain $u\in\ho_{\gamma}(\omega)\subset\ho_{\gamma_{D}}(\omega)$ 
with the stability estimate
$$\norm{\na u}_{0,\omega}\leq\kappa\norm{g}_{0,\omega},\qquad
\kappa=\min\big\{\kappa(\omega,\gamma_{D}),\kappa(\omega,\gamma)\big\},$$
see \eqref{defkappaupperbound}. Thus $v\in\ho_{-1,\ga_D}(\om)$ solves $\div v=f$. 
It remains to show the estimates. Using
\begin{align*}
\norm{\na\eta\cdot v^{\top}}_{0,\reals^{d}},
\norm{\na\eta\cdot v}_{0,\reals^{d}}
&\leq\frac{\xi'_{\infty}}{r_{2}-r_{1}}
\norm{v}_{0,B_{r_{2}}\setminus\overline{B}_{r_{1}}}
\end{align*}
and by \eqref{navf} we compute
\begin{align*}
\bnorm{\na(\eta v)}_{0,\om}
&\leq\underbrace{\norm{\na v}_{0,\reals^{d}}}_{=\norm{f}_{0,\om}}
+\norm{\na\eta\,v^{\top}}_{0,\reals^{d}},\\
\norm{\na v_{\omega}}_{0,\om}
&=\norm{\na u}_{0,\omega}
\leq\kappa(\omega,\gamma_{D})\norm{g}_{0,\omega}
\leq\kappa(\omega,\gamma_{D})\big(\norm{f}_{0,\om}
+\norm{\na\eta\cdot v}_{0,\reals^{d}}\big),\\
\norm{v}_{0,B_{r_{2}}\setminus\overline{B}_{r_{1}}}
&\leq\rho(r_{2})\norm{v}_{-1,\reals^{d}}
\leq c_{d}\rho(r_{2})\underbrace{\norm{\na v}_{0,\reals^{d}}}_{=\norm{f}_{0,\om}},
\end{align*}
which finally proves $\norm{\na v}_{0,\om}\leq\hat{\kappa}(\om,\ga_{D})\norm{f}_{0,\om}$,
finishing the proof.
\end{proof}

\section{Applications for Exterior Domains}

\subsection{Inf-Sup Lemma and Estimates of the Distance to Solenoidal Fields}

A direct consequence of Lemma \ref{stablemextdom} is an estimate for the distance of vector fields
to solenoidal fields:

\begin{cor}[distance lemma for exterior domains]
\label{distlemextdom}
For any $v\in\ho_{-1,\ga_{D}}(\om)$ there exists a solenoidal
$v_{0}\in\s_{-1,\ga_{D}}(\om)$ such that
$$\dist\big(v,\s_{-1,\ga_{D}}(\om)\big)
=\inf_{\phi\in\s_{-1,\ga_{D}}(\om)}\bnorm{\na(v-\phi)}_{0,\om}
\leq\bnorm{\na(v-v_{0})}_{0,\om}
\leq\kappa(\om,\ga_{D})\norm{\div v}_{0,\om}.$$
\end{cor}

\begin{proof}
For $v\in\ho_{-1,\ga_{D}}(\om)$ solve $\div\tilde{v}=\div v$ with $\tilde{v}\in\ho_{-1,\ga_{D}}(\om)$
and the stability estimate
$\norm{\na\tilde{v}}_{0,\om}\leq\kappa(\om,\ga_{D})\norm{\div v}_{0,\om}$
by Lemma \ref{stablemextdom}. Then $v_{0}:=v-\tilde{v}\in\s_{-1,\ga_{D}}(\om)$ and we have
$\bnorm{\na(v-v_{0})}_{0,\om}
=\bnorm{\na\tilde{v}}_{0,\om}
\leq\kappa(\om,\ga_{D})\norm{\div v}_{0,\om}$.
\end{proof}

\begin{cor}[inhomogeneous distance lemma for exterior domains]
\label{inhomodistlemextdom}
For any $v\in\ho_{-1}(\om)$ there exists 
a solenoidal $v_{0}\in\s_{-1}(\om)$ such that
$v_{0}-v\in\ho_{-1,\ga_{D}}(\om)$, i.e., $v_{0}|_{\ga_{D}}=v|_{\ga_{D}}$, and
$\bnorm{\na(v_{0}-v)}_{0,\om}
\leq\kappa(\om,\ga_{D})\norm{\div v}_{0,\om}$.
\end{cor}

\begin{proof}
For $v\in\ho_{-1}(\om)$ we solve by Lemma \ref{stablemextdom}
$\div\tilde{v}=\div v$ with some $\tilde{v}\in\ho_{-1,\ga_{D}}(\om)$
and $\norm{\na\tilde{v}}_{0,\om}\leq\kappa(\om,\ga_{D})\norm{\div v}_{0,\om}$. 
Then $v_{0}:=v-\tilde{v}\in\s(\om)$ with $v-v_{0}=\tilde{v}\in\ho_{\ga_{D}}(\om)$ and
$\bnorm{\na(v_{0}-v)}_{0,\om}
=\bnorm{\na\tilde{v}}_{0,\om}
\leq\kappa(\om,\ga_{D})\norm{\div v}_{0,\om}$.
\end{proof}

As in the case of a bounded domain, Corollary \ref{inhomodistlemextdom}
can be seen as a lifting lemma, lifting the boundary datum
$v|_{\ga_{D}}$ to the domain $\om$,
in this case with a solenoidal representative. 
By solving $g=\div v$ Lemma \ref{stablemextdom}
yields immediately also the following inf-sup result:

\begin{cor}[inf-sup lemma for exterior domains]
\label{infsuplemextdom}
It holds
$$\inf_{f\in\lt(\om)}\sup_{v\in\ho_{-1,\ga_{D}}(\om)}
\frac{\scp{f}{\div v}_{0,\om}}{\norm{f}_{0,\om}\norm{\na v}_{0,\om}}\geq\frac{1}{\kappa(\om,\ga_{D})}.$$
\end{cor}

\subsection{Solution Theory for the Stationary Stokes System}

For $\nu>0$, $F\in\lt_{1}(\om)$, $v_{D}\in\s_{-1}(\om)$
a solution theory for the stationary Stokes problem follows.
The equations or first resp. second order systems 
are the same as in the case of a bounded domain, e.g.,
\begin{center}
\begin{minipage}{60mm}
\begin{align*}
-\Div\sigma
&=F
&
&\text{in }\om,\\
\sigma
&=\nu\na v-p\,\I
&
&\text{in }\om,\\
-\div v
&=0
&
&\text{in }\om,\\
v
&=v_{D}
&
&\text{on }\ga_{D},\\
\sigma n
&=0
&
&\text{on }\ga_{N}
\end{align*}
\end{minipage}
\end{center}
(for simplicity we assume again $\sigma_{N}=0$) resp.
\begin{center}
\begin{minipage}{60mm}
\begin{align*}
-\nu\Delta v+\na p
&=F
&
&\text{in }\om,\\
-\div v
&=0
&
&\text{in }\om,\\
v
&=v_{D}
&
&\text{on }\ga_{D},\\
(\nu\na v-p)n
&=0
&
&\text{on }\ga_{N}
\end{align*}
\end{minipage}
\end{center}
with additional proper decay conditions at infinity $v\in\lt_{-1}(\om)$ and $\na v\in\lt(\om)$
which read in classical point wise terms (more vaguely) as
$$v(x)\xrightarrow{|x|\to\infty}0.$$
Note that \eqref{meanbcbddom} is no longer a necessary condition in the case of an exterior domain
due to the possible lack of integrability.
Therefore, our lifting lemma Corollary \ref{inhomodistlemextdom}
does not need an additional assumption on $\div v$
as in Corollary \ref{inhomodistlembddom}.
Let us assume a slightly more general viscosity\footnote{The viscosity 
$\nu$ can even be assumed to be a bounded, positive definite, symmetric tensor field.
Moreover,  we note that $\nu_{-}|T|^2\leq|\nu^{1/2}T|^2=\nu T:T\leq\nu_{+}|T|^2$
and thus also $\nu_{-}|\nu^{-1/2}T|^2\leq|T|^2\leq\nu_{+}|\nu^{-1/2}T|^2$.}
$\nu\in\li(\om)$,
bounded from below and above by two positive constants
$\nu_{-}$ amd $\nu_{+}$, respectively.
A possible variational formulation (see, e.g., \cite{La,GiSe}) is given by the following:
Find $v\in v_{D}+\s_{-1,\ga_{D}}(\om)$ such that for all $\phi\in\s_{-1,\ga_{D}}(\om)$
$$\scp{\nu\na v}{\na\phi}_{0,\om}
=\scp{F}{\phi}_{0,\om}.$$
Note that by $\scp{F}{\phi}_{0,\om}=\bscp{\rho\,F}{\rho^{-1}\phi}_{0,\om}$ 
the right hand side is well defined.
Using the ansatz $v=v_{D}+\hat{v}$ with $\hat{v}\in\s_{-1,\ga_{D}}(\om)$ 
we reduce this formulation to find $\hat{v}\in\s_{-1,\ga_{D}}(\om)$ 
such that for all $\phi\in\s_{-1,\ga_{D}}(\om)$
$$\scp{\nu\na\hat{v}}{\na\phi}_{0,\om}
=\scp{F}{\phi}_{0,\om}
-\scp{\nu\na v_{D}}{\na\phi}_{0,\om}.$$
Again, another formulation taking the pressure into account 
and removing the unpleasant solenoidal condition from the Hilbert space
is the following saddle point formulation: 
Find $(\hat{v},p)\in\ho_{-1,\ga_{D}}(\om)\times\lt(\om)$
such that for all $(\phi,\varphi)\in\ho_{-1,\ga_{D}}(\om)\times\lt(\om)$
\begin{align}
\label{spformextstokes}
\begin{split}
\scp{\nu\na\hat{v}}{\na\phi}_{0,\om}
-\scp{p}{\div\phi}_{0,\om}
&=\scp{F}{\phi}_{0,\om}
-\scp{\nu\na v_{D}}{\na\phi}_{0,\om},\\
-\scp{\div\hat{v}}{\varphi}_{0,\om}
&=0.
\end{split}
\end{align}

\begin{cor}[Stokes lemma for exterior domains]
\label{stokeslemextdom}
For $\nu$, $F\in\lt_{1}(\om)$, $v_{D}\in\s_{-1}(\om)$
the Stokes system is uniquely solvable with 
$v=v_{D}+\hat{v}\in v_{D}+\s_{-1,\ga_{D}}(\om)\subset\s_{-1}(\om)$
and $p\in\lt(\om)$. Moreover, 
\begin{align*}
\nu\norm{\na\hat{v}}_{0,\om}
&\leq c_{fp}(\om,\ga_{D})\norm{F}_{1,\om}
+\nu\norm{\na v_{D}}_{0,\om},\\
\nu\norm{\na v}_{0,\om}
&\leq c_{fp}(\om,\ga_{D})\norm{F}_{1,\om}
+2\nu\norm{\na v_{D}}_{0,\om},\\
\norm{p}_{0,\om}
&\leq 2\kappa(\om,\ga_{D})\big(c_{fp}(\om,\ga_{D})\norm{F}_{1,\om}
+\nu\norm{\na v_{D}}_{0,\om}\big).
\end{align*}
\end{cor}

\begin{proof}
Standard saddle point theory and the inf-sup lemma, Corollary \ref{infsuplemextdom},
shows existence and the estimates follow by standard arguments,
which provide also uniqueness.
Note that by the Friedrichs/Poincar\'e estimates in Lemma \ref{fplemextdom}
the principal part of the bilinear form is positive
over $\ho_{-1,\ga_{D}}(\om)$, and that we solve $p=\div\phi$ by Lemma \ref{stablemextdom}
to get the estimates for the pressure $p$.
\end{proof}

\subsection{A Posteriori Error Estimates for Stationary Stokes Equations}

Before proceeding, we need one more polynomial weighted Sobolev space.
For this, we recall $\Div$ acting as usual row wise on $\reals^{d\times d}$-tensor fields and define
$$\tilde{\d}(\om)
:=\setb{\Theta\in\lt(\om)}{\Div\Theta\in\lt_{1}(\om)},\qquad
\tilde{\d}_{\ga_{N}}(\om),$$
where $\tilde{\d}_{\ga_{N}}(\om)$ is the closure of 
$\ci_{\ga_{N}}(\om)$-tensor fields in the norm of the Sobolev space $\tilde{\d}(\om)$.
Then we observe for all $\phi\in\ho_{-1,\ga_{D}}(\om)$ and all $\tau\in\tilde{\d}_{\ga_{N}}(\om)$
\begin{align}
\label{partinttildeD}
\scp{\tau}{\na\phi}_{0,\om}=-\scp{\Div\tau}{\phi}_{0,\om}.
\end{align}
Note that the right hand side is well defined since 
$\scp{\Div\tau}{\phi}_{0,\om}=\scp{\rho\Div\tau}{\rho^{-1}\phi}_{0,\om}$.

From now on we assume that we have approximations 
$$\tilde{v}\in\lt_{-1}(\om),\qquad
\tilde{p}\in\lt(\om),\qquad
\tilde{T}\in\lt(\om),\qquad
\tilde{\sigma}\in\lt(\om)$$
of our exact solutions from \eqref{spformextstokes} and Corollary \ref{stokeslemextdom}
\begin{align*}
v=v_{D}+\hat{v}&\in v_{D}+\s_{-1,\ga_{D}}(\om)\subset\s_{-1}(\om),
&
T&:=\na v\in\lt(\om),\\
p&\in\lt(\om),
&
\sigma&\;=\nu\na v-p\,\I\in\lt(\om),
\end{align*}
respectively, for given data $\nu$, $F\in\lt_{1}(\om)$, and $v_{D}\in\s_{-1}(\om)$. 
We recall from \eqref{spformextstokes} that
$(v,p)$ solves for all $\phi\in\ho_{-1,\ga_{D}}(\om)$
\begin{align}
\label{spform2extstokes}
\scp{\nu\na v}{\na\phi}_{0,\om}
-\scp{p}{\div\phi}_{0,\om}
&=\scp{F}{\phi}_{0,\om}.
\end{align}

\subsubsection{A Posteriori Estimates for the Velocity Field: Solenoidal Approximations}

First, we assume the simplest case that 
$$\tilde{T}=\na\tilde{v},\qquad
\tilde{v}\in v_{D}+\s_{-1,\ga_{D}}(\om)\subset\s_{-1}(\om),$$
i.e., $\tilde{v}-v_{D}\in\s_{-1,\ga_{D}}(\om)$.
Then by \eqref{spform2extstokes} we have for all solenoidal $\phi\in\s_{-1,\ga_{D}}(\om)$
$$\bscp{\nu\na(v-\tilde{v})}{\na\phi}_{0,\om}
=\scp{F}{\phi}_{0,\om}
-\scp{\nu\na\tilde{v}}{\na\phi}_{0,\om}.$$
Let $\tau\in\tilde{\d}_{\ga_{N}}(\om)$ and $q\in\lt(\om)$.
Using \eqref{partinttildeD} 
and $\scp{q\,\I}{\na\phi}_{0,\om}=0$ 
\big(actually this holds for all $\phi\in\s_{-1}(\om)$ since $\I:\na\phi=\div\phi$\big) 
as well as the Friedrichs/Poincar\'e estimate from Lemma \ref{fplemextdom} we compute
\begin{align}
\label{apostestcompsol}
\begin{split}
&\qquad\bscp{\nu\na(v-\tilde{v})}{\na\phi}_{0,\om}\\
&=\scp{\Div\tau+F}{\phi}_{0,\om}
+\scp{\tau+q\,\I-\nu\na\tilde{v}}{\na\phi}_{0,\om}\\
&\leq\norm{\Div\tau+F}_{1,\om}
\norm{\phi}_{-1,\om}
+\bnorm{\nu^{-1/2}(\tau+q\,\I-\nu\na\tilde{v})}_{0,\om}
\norm{\nu^{1/2}\na\phi}_{0,\om}\\
&\leq\Big(\nu_{-}^{-1/2}c_{fp}(\om,\ga_{D})\norm{\Div\tau+F}_{1,\om}
+\bnorm{\nu^{-1/2}(\tau+q\,\I-\nu\na\tilde{v})}_{0,\om}
\Big)\norm{\nu^{1/2}\na\phi}_{0,\om}.
\end{split}
\end{align}
Choosing $\phi_{v}:=v-\tilde{v}=\hat{v}+v_{D}-\tilde{v}\in\s_{-1,\ga_{D}}(\om)$ 
shows a first a posteriori estimate:

\begin{theo}[a posteriori error estimate I for exterior domains]
\label{aposttheo1extdom}
Let $\tilde{v}\in v_{D}+\s_{-1,\ga_{D}}(\om)$.
Then for all $\tau\in\tilde{\d}_{\ga_{N}}(\om)$ and all $q\in\lt(\om)$ it holds
$$\norm{\nu^{1/2}\na(v-\tilde{v})}_{0,\om}
\leq\nu_{-}^{-1/2}c_{fp}(\om,\ga_{D})\norm{\Div\tau+F}_{1,\om}
+\bnorm{\nu^{-1/2}(\tau+q\,\I-\nu\na\tilde{v})}_{0,\om}.$$
\end{theo}

The upper bound coincides with 
the norm of the error on the left hand side,
if $\tau=\sigma$ (i.e., $\tau$ coincides with the exact stress tensor)
and $q=p$ (i.e., $q$ represents the exact pressure $p$), i.e., we have
\begin{align*}
&\qquad\norm{\nu^{1/2}\na(v-\tilde{v})}_{0,\om}\\
&=\min_{\substack{\tau\in\tilde{\d}_{\ga_{N}}(\om),\\q\in\lt(\om)}}
\Big(\nu_{-}^{-1/2}c_{fp}(\om,\ga_{D})\norm{\Div\tau+F}_{1,\om}
+\bnorm{\nu^{-1/2}(\tau+q\,\I-\nu\na\tilde{v})}_{0,\om}\Big)
\end{align*}
and the minimum is attained at $(\tau,q)=(\sigma,p)$.
However, Theorem \ref{aposttheo1extdom} has a drawback: 
The estimate is valid only for those approximate vector fields $\tilde{v}$, 
which exactly satisfy the solenoidal condition and the boundary condition.
In practice, the solenoidal requirement is difficult to fulfill and approximations arising
in `real life' computations often satisfy the solenoidal condition only
approximately. Therefore, our next goal is to extend the estimate 
to a wider class of non-solenoidal vector fields.

\subsubsection{A Posteriori Estimates for the Velocity Field: Non-Solenoidal Approximations}

Now we assume only
$$\tilde{T}=\na\tilde{v},\qquad
\tilde{v}\in v_{D}+\ho_{-1,\ga_{D}}(\om)\subset\ho_{-1}(\om),$$
i.e., $\tilde{v}-v_{D}\in\ho_{-1,\ga_{D}}(\om)$, this is
$\tilde{v}$ is not solenoidal but satisfies the boundary condition exactly.
Utilizing the stability lemma, Lemma \ref{stablemextdom},
there exists $w\in\ho_{-1,\ga_D}(\om)$ such that
$\div w=-\div\tilde{v}$ and $\norm{\na w}_{0,\om}\leq\kappa(\om,\ga_{D})\norm{\div\tilde{v}}_{0,\om}$.
Then $\tilde{v}_{0}:=\tilde{v}+w\in v_{D}+\s_{-1,\ga_D}(\om)$ 
and by Theorem \ref{aposttheo1extdom}
\begin{align}
\nonumber
\bnorm{\nu^{1/2}\na(v-\tilde{v})}_{0,\om}
&\leq\bnorm{\nu^{1/2}\na(v-\tilde{v}_{0})}_{0,\om}
+\norm{\nu^{1/2}\na w}_{0,\om}\\
\nonumber
&\leq\nu_{-}^{-1/2}c_{fp}(\om,\ga_{D})\norm{\Div\tau+F}_{1,\om}
+\bnorm{\nu^{-1/2}(\tau+q\,\I-\nu\na\tilde{v}_{0})}_{0,\om}\\
\label{apostestcompnonsol}
&\qquad+\norm{\nu^{1/2}\na w}_{0,\om}\\
\nonumber
&\leq\nu_{-}^{-1/2}c_{fp}(\om,\ga_{D})\norm{\Div\tau+F}_{1,\om}
+\bnorm{\nu^{-1/2}(\tau+q\,\I-\nu\na\tilde{v})}_{0,\om}\\
\nonumber
&\qquad+2\norm{\nu^{1/2}\na w}_{0,\om}.
\end{align}

\begin{theo}[a posteriori error estimate II for exterior domains]
\label{aposttheo2extdom}
Let $\tilde{v}\in v_{D}+\ho_{-1,\ga_{D}}(\om)$.
Then for all $\tau\in\tilde{\d}_{\ga_{N}}(\om)$ and all $q\in\lt(\om)$ it holds
\begin{align*}
\norm{\nu^{1/2}\na(v-\tilde{v})}_{0,\om}
&\leq\nu_{-}^{-1/2}c_{fp}(\om,\ga_{D})\norm{\Div\tau+F}_{1,\om}
+\bnorm{\nu^{-1/2}(\tau+q\,\I-\nu\na\tilde{v})}_{0,\om}\\
&\qquad+2\nu_{+}^{1/2}\kappa(\om,\ga_{D})\norm{\div\tilde{v}}_{0,\om}.
\end{align*}
\end{theo}

If the approximation $\tilde{v}$ is solenoidal we get back Theorem \ref{aposttheo1extdom}
and, again, the upper bound coincides with the norm of the error on the left hand side
if $\tau=\sigma$, $q=p$.
If the approximation $\tilde{v}$ is solenoidal just in, e.g., $\reals^{d}\setminus\overline{B}_{r_{2}}$
then we get trivially an estimate by Theorem \ref{aposttheo2extdom}, replacing 
the term $\norm{\div\tilde{v}}_{0,\om}$ by $\norm{\div\tilde{v}}_{0,\omega}$.
But with a moderate additional assumption on the decay of the approximation 
we can even do better in this case, replacing the stability constant $\kappa(\om,\ga_{D})$
by a stability constant of the bounded domain $\omega$.
For this let $\tilde{v}=v_{D}+w\in v_{D}+\ho_{-1,\ga_{D}}(\om)$
with $\div\tilde{v}=\div w=0$ in $\reals^{d}\setminus\overline{B}_{r_{2}}$ and additionally,
if $\gamma_{D}=\gamma$, i.e., $\ga_{D}=\ga$,
\begin{align}
\label{pwdecay}
|w|\leq c\,r^{-m},\qquad
m>d-1
\end{align}
for $r\to\infty$ with some $c>0$ independent of $r$.
Note that for $r^{-m}\in\lt_{-1}(\reals^{d}\setminus\overline{B}_{1})$ 
it is sufficient that $m>d/2-1$.
We consider the ansatz
$$\tilde{v}_{0}:=\tilde{v}+\begin{cases}u&\text{in }\omega,\\0&\text{in }\reals^{d}\setminus\overline{B}_{r_{2}},\end{cases}$$
with $u\in\ho_{\gamma_{D}}(\omega)$ and $\div u=-\div\tilde{v}$ in $\omega$. 
Utilizing Lemma \ref{stablembddom} we find such a $u$ together with the stability estimate
$\norm{\na u}_{0,\omega}\leq\kappa(\omega,\gamma_{D})\norm{\div\tilde{v}}_{0,\omega}$,
provided that in the case $\gamma_{D}=\gamma$, i.e., $\ga_{D}=\ga$, additionally
$\div\tilde{v}\in\lt_{\bot}(\omega)$ holds. For this we notice (for $\ga_{D}=\ga$)
that for any $r>r_{2}$
$$\big|\int_{\omega}\div\tilde{v}\big|
=\big|\int_{\omega}\div w\big|
=\big|\int_{S_{r}}n\cdot w\big|
\leq c\,r^{d-1-m}\xrightarrow{r\to\infty}0.$$
Therefore, $\tilde{v}_{0}\in v_{D}+\s_{-1,\ga_{D}}(\om)$
is an admissible vector field for \eqref{apostestcompnonsol} showing
\begin{align*}
\norm{\nu^{1/2}\na(v-\tilde{v})}_{0,\om}
&\leq\nu_{-}^{-1/2}c_{fp}(\om,\ga_{D})\norm{\Div\tau+F}_{1,\om}
+\bnorm{\nu^{-1/2}(\tau+q\,\I-\nu\na\tilde{v})}_{0,\om}\\
&\qquad+2\nu_{+}^{1/2}\underbrace{\norm{\na(\tilde{v}_{0}-\tilde{v})}_{0,\om}}_{=\norm{\na u}_{0,\omega}}.
\end{align*}
Hence we have the following:

\begin{cor}[a posteriori error estimate III for exterior domains]
\label{apostcor3extdom}
Let $\tilde{v}\in v_{D}+\ho_{-1,\ga_{D}}(\om)$
with $\div\tilde{v}=0$ in $\reals^{d}\setminus\overline{B}_{r_{2}}$ and, if $\ga_{D}=\ga$, \eqref{pwdecay}.
Then for all $\tau\in\tilde{\d}_{\ga_{N}}(\om)$ and all $q\in\lt(\om)$
\begin{align*}
\norm{\nu^{1/2}\na(v-\tilde{v})}_{0,\om}
&\leq\nu_{-}^{-1/2}c_{fp}(\om,\ga_{D})\norm{\Div\tau+F}_{1,\om}
+\bnorm{\nu^{-1/2}(\tau+q\,\I-\nu\na\tilde{v})}_{0,\om}\\
&\qquad+2\nu_{+}^{1/2}\kappa(\omega,\gamma_{D})\norm{\div\tilde{v}}_{0,\omega}.
\end{align*}
\end{cor}

Here the last term on the right hand side is a penalty for possible violation
of the solenoidal condition in $\omega$.

\subsubsection{A Posteriori Estimates for the Pressure Function}

By Lemma \ref{stablemextdom} there exists a vector field $\phi_{p}\in\ho_{-1,\ga_D}(\om)$ with
$\div\phi_{p}=p-\tilde{p}$ 
and $\norm{\na\phi_{p}}_{0,\om}\leq\kappa(\om,\ga_{D})\norm{p-\tilde{p}}_{0,\om}$.
\eqref{spform2extstokes} implies for all $\psi\in v_{D}+\ho_{-1,\ga_{D}}(\om)$
and all $\tau\in\tilde{\d}_{\ga_{N}}(\om)$
\begin{align*}
&\qquad\norm{p-\tilde{p}}_{0,\om}^2
=\scp{p-\tilde{p}}{\div\phi_{p}}_{0,\om}\\
&=\scp{\nu\na(v-\psi)}{\na\phi_{p}}_{0,\om}
-\scp{\Div\tau+F}{\phi_{p}}_{0,\om}
+\scp{\nu\na\psi-\tilde{p}\,\I-\tau}{\na\phi_{p}}_{0,\om}\\
&\leq\Big(\norm{\nu\na(v-\psi)}_{0,\om}
+c_{fp}(\om,\ga_{D})\norm{\Div\tau+F}_{1,\om}
+\norm{\nu\na\psi-\tilde{p}\,\I-\tau}_{0,\om}\Big)
\norm{\na\phi_{p}}_{0,\om},
\end{align*}
where we have used Lemma \ref{fplemextdom} for $\phi_{p}$
and the equation $\div\phi_{p}=\I:\na\phi_{p}$.
Therefore, we obtain
\begin{align*}
\norm{p-\tilde{p}}_{0,\om}
&\leq\kappa(\om,\ga_{D})
\Big(\nu_{+}^{1/2}\norm{\nu^{1/2}\na(v-\psi)}_{0,\om}
+c_{fp}(\om,\ga_{D})\norm{\Div\tau+F}_{1,\om}\\
&\qquad+\nu_{+}^{1/2}\bnorm{\nu^{-1/2}(\tau+\tilde{p}\,\I-\nu\na\psi)}_{0,\om}\Big)
\end{align*}
and by Theorem \ref{aposttheo2extdom} with $\tilde{v}:=\psi$, $q:=\tilde{p}$ we get:

\begin{theo}[a posteriori error estimate IV for exterior domains]
\label{aposttheo4extdom}
Let $\tilde{p}\in\lt(\om)$.
Then for all $\tau\in\tilde{\d}_{\ga_{N}}(\om)$ and all $\psi\in v_{D}+\ho_{-1,\ga_{D}}(\om)$ it holds
\begin{align*}
\norm{p-\tilde{p}}_{0,\om}
&\leq\kappa(\om,\ga_{D})
\Big((\nu_{-}^{-1/2}\nu_{+}^{1/2}+1)c_{fp}(\om,\ga_{D})\norm{\Div\tau+F}_{1,\om}\\
&\qquad+2\nu_{+}^{1/2}\bnorm{\nu^{-1/2}(\tau+\tilde{p}\,\I-\nu\na\psi)}_{0,\om}
+2\nu_{+}\kappa(\om,\ga_{D})\norm{\div\psi}_{0,\om}\Big).
\end{align*}
\end{theo}

The upper bound consists of the same terms as the upper bound 
of Theorem \ref{aposttheo2extdom} and vanishes if
$\psi=v$, $\tau=\sigma$, $\tilde{p}=p$.
However, in this case, the quantity (error measure) on the left hand side depends, e.g., 
on the stability constant $\kappa(\om,\ga_{D})$.

\subsubsection{A Posteriori Estimates for Non-Conforming Approximations}

Let us now assume that we have a very non-conforming approximation
of the strain tensor field 
$$T:=\na v,\qquad
v=v_{D}+\hat{v}\in v_{D}+\s_{-1,\ga_{D}}(\om)\subset\s_{-1}(\om),$$ 
given just by some $\tilde{T}\in\lt(\om)$.
An example could be a broken gradient tensor field as output
of some discontinous Galerkin method.
By the triangle inequality, i.e.,
$$\norm{\nu^{1/2}(T-\tilde{T})}_{0,\om}
\leq\norm{\nu^{1/2}\na(v-\psi)}_{0,\om}
+\norm{\nu^{1/2}(\na\psi-\tilde{T})}_{0,\om},$$
and Theorem \ref{aposttheo2extdom} ($\tilde{v}=\psi$),
and Theorem \ref{aposttheo4extdom} (and again triangle inequality)
we obtain the following result:

\begin{theo}[a posteriori error estimate V for exterior domains]
\label{aposttheo5extdom}
Let $\tilde{T}\in\lt(\om)$ and $\tilde{p}\in\lt(\om)$.
Then for all $\psi\in v_{D}+\ho_{-1,\ga_{D}}(\om)$,
all $\tau\in\tilde{\d}_{\ga_{N}}(\om)$, and all $q\in\lt(\om)$ it holds
\begin{align*}
\norm{\nu^{1/2}(T-\tilde{T})}_{0,\om}
&\leq\nu_{-}^{-1/2}c_{fp}(\om,\ga_{D})\norm{\Div\tau+F}_{1,\om}
+\bnorm{\nu^{-1/2}(\tau+q\,\I-\nu\na\psi)}_{0,\om}\\
&\qquad+2\nu_{+}^{1/2}\kappa(\om,\ga_{D})\norm{\div\psi}_{0,\om}
+\norm{\nu^{1/2}(\na\psi-\tilde{T})}_{0,\om}\\
&\leq\nu_{-}^{-1/2}c_{fp}(\om,\ga_{D})\norm{\Div\tau+F}_{1,\om}
+\bnorm{\nu^{-1/2}(\tau+q\,\I-\nu\tilde{T})}_{0,\om}\\
&\qquad+2\nu_{+}^{1/2}\kappa(\om,\ga_{D})\norm{\div\psi}_{0,\om}
+2\norm{\nu^{1/2}(\na\psi-\tilde{T})}_{0,\om},\\
\norm{p-\tilde{p}}_{0,\om}
&\leq\kappa(\om,\ga_{D})
\Big((\nu_{-}^{-1/2}\nu_{+}^{1/2}+1)c_{fp}(\om,\ga_{D})\norm{\Div\tau+F}_{1,\om}\\
&\qquad+2\nu_{+}^{1/2}\bnorm{\nu^{-1/2}(\tau+\tilde{p}\,\I-\nu\tilde{T})}_{0,\om}
+2\nu_{+}\kappa(\om,\ga_{D})\norm{\div\psi}_{0,\om}\\
&\qquad\qquad+2\nu_{+}^{1/2}\norm{\nu^{1/2}(\na\psi-\tilde{T})}_{0,\om}\Big).
\end{align*}
\end{theo}

For $\tilde{T}=\na\tilde{v}$, $\psi=\tilde{v}\in v_{D}+\ho_{-1,\ga_{D}}(\om)$
we get back Theorem \ref{aposttheo2extdom} and Theorem \ref{aposttheo4extdom}.
Let us investigate the latter summands a bit closer and identify them
in terms of parts of the error.
For this, we use the well known (row wise) Helmholtz decomposition,
see \eqref{helmmixedbc} of the Appendix, and decompose\footnote{Note that
the decomposition is orthogonal with respect to the weighted
$\scp{\nu\,\cdot\,}{\,\cdot\,}_{0,\om}$-inner product.} 
the error according to
$$T-\tilde{T}=\na w+\tilde{T}_{0}\in\na\ho_{-1,\ga_{D}}(\om)\oplus_{0,\nu}\nu^{-1}{}_{0}\d_{\ga_{N}}(\om).$$
Note that due to orthogonality
\begin{align}
\label{TtTortho}
\bnorm{\nu^{1/2}(T-\tilde{T})}_{0,\om}^2
=\norm{\nu^{1/2}\na w}_{0,\om}^2
+\norm{\nu^{1/2}\tilde{T}_{0}}_{0,\om}^2
\end{align}
and that $T-\na v_{D}=\na\hat{v}\in\na\s_{-1,\ga_{D}}(\om)\subset\na\ho_{-1,\ga_{D}}(\om)$ 
already belongs to the first space. 
Hence the latter decomposition is actually a decomposition of $\na v_{D}-\tilde{T}$,
more precisely $\na v_{D}-\tilde{T}=T-\na\hat{v}-\tilde{T}=\na(w-\hat{v})+\tilde{T}_{0}$.
For the second error part $\tilde{T}_{0}$ we observe 
for all $\phi\in\ho_{-1,\ga_{D}}(\om)$ by orthogonality
$$\norm{\nu^{1/2}\tilde{T}_{0}}_{0,\om}^2
=\scp{T-\tilde{T}}{\nu\tilde{T}_{0}}_{0,\om}
=\bscp{\na(v_{D}+\phi)-\tilde{T}}{\nu\tilde{T}_{0}}_{0,\om}$$
and thus 
$\norm{\nu^{1/2}\tilde{T}_{0}}_{0,\om}
\leq\bnorm{\nu^{1/2}(\na(v_{D}+\phi)-\tilde{T})}_{0,\om}$.
In other words,
\begin{align}
\label{upbdTz}
\begin{split}
\norm{\nu^{1/2}\tilde{T}_{0}}_{0,\om}
&=\min_{\phi\in\ho_{-1,\ga_{D}}(\om)}\bnorm{\nu^{1/2}(\na(v_{D}+\phi)-\tilde{T})}_{0,\om}\\
&=\min_{\psi\in v_{D}+\ho_{-1,\ga_{D}}(\om)}\bnorm{\nu^{1/2}(\na\psi-\tilde{T})}_{0,\om}
\end{split}
\end{align}
and the minima are attained at 
$$\hat{\phi}=\hat{v}-w\in\ho_{-1,\ga_{D}}(\om),\qquad
\hat{\psi}=v_{D}+\hat{\phi}=v-w\in v_{D}+\ho_{-1,\ga_{D}}(\om),$$
as $\na\hat{\psi}-\tilde{T}=T-\na w-\tilde{T}=\tilde{T}_{0}$.
Therefore, the minima of the last terms on the right hand sides in Theorem \ref{aposttheo5extdom} 
equal the error part $\norm{\nu^{1/2}\tilde{T}_{0}}_{0,\om}$.

\subsubsection{A Posteriori Estimates for the Stress Tensor Field}

Error estimates for the stress tensor field follow immediately 
from the above derived estimates for the velocity vector field and the pressure function.
Indeed, let $\tilde{\sigma}\in\lt(\om)$ 
be an approximation of the exact stress tensor $\sigma=\nu\na v-p\,\I=\nu T-p\,\I$. 
Moreover, let $\tilde{T}\in\lt(\om)$ and $\tilde{p}\in\lt(\om)$. 
Then, the respective error is simply subject to the triangle inequality
\begin{align*}
\norm{\tilde{\sigma}-\sigma}_{0,\om}
\leq\norm{\tilde{\sigma}-\nu\tilde{T}+\tilde{p}\,\I}_{0,\om}
+\nu_{+}^{1/2}\norm{\nu^{1/2}(T-\tilde{T})}_{0,\om}
+d^{1/2}\norm{p-\tilde{p}}_{0,\om},
\end{align*}
where we can also put $\tilde{T}=\na\tilde{v}$, $\tilde{v}\in v_{D}+\ho_{-1,\ga_{D}}(\om)$.
The first term on the right hand side contains only known tensor fields
and the second and third ones are estimated by, e.g.,
Theorem \ref{aposttheo2extdom}, Theorem \ref{aposttheo4extdom}, and Theorem \ref{aposttheo5extdom}.

\subsubsection{Lower Bounds for the Velocity Field}

Let $\tilde{v}\in v_{D}+\ho_{-1,\ga_{D}}(\om)$, i.e.,
$v-\tilde{v}\in\ho_{-1,\ga_{D}}(\om)$.
Obviously, (as the subsequent $\max$-property holds for any Hilbert\footnote{In any
Hilbert space $\mathsf{H}$ it holds 
$|x|^2=\max_{y\in\mathsf{H}}\big(2\scp{x}{y}-|y|^2\big)$.
Here $\mathsf{H}=\na\ho_{-1,\ga_{D}}(\om)$.} space) 
we have by \eqref{spform2extstokes}
\begin{align*}
\norm{\nu^{1/2}\na(v-\tilde{v})}_{0,\om}^2
&=\max_{\phi\in\ho_{-1,\ga_{D}}(\om)}
\big(2\bscp{\nu\na(v-\tilde{v})}{\na\phi}_{0,\om}
-\norm{\nu^{1/2}\na\phi}_{0,\om}^2\big)\\
&\geq2\scp{\nu\na v}{\na\phi}_{0,\om}
-2\scp{\nu\na\tilde{v}}{\na\phi}_{0,\om}
-\norm{\nu^{1/2}\na\phi}_{0,\om}^2\\
&=2\scp{F}{\phi}_{0,\om}
+2\scp{q}{\div\phi}_{0,\om}
-\bscp{\nu\na(2\tilde{v}+\phi)}{\na\phi}_{0,\om}\\
&\qquad+2\scp{p-q}{\div\phi}_{0,\om}
\end{align*}
and the maximum is attained at $\phi=v-\tilde{v}\in\ho_{-1,\ga_{D}}(\om)$.
The last term can simply and roughly be estimated by Theorem \ref{aposttheo4extdom} ($\tilde{p}=q$)
showing the following result.

\begin{theo}[a posteriori error estimate VI for exterior domains]
\label{aposttheo6extdom}
Let $\tilde{v}\in v_{D}+\ho_{-1,\ga_{D}}(\om)$.
Then for all $\phi\in\ho_{-1,\ga_{D}}(\om)$, all $\tau\in\tilde{\d}_{\ga_{N}}(\om)$,
all $\psi\in v_{D}+\ho_{-1,\ga_{D}}(\om)$, and all $q\in\lt(\om)$
\begin{align*}
&\qquad\norm{\nu^{1/2}\na(v-\tilde{v})}_{0,\om}^2\\
&\geq2\scp{F}{\phi}_{0,\om}
+2\scp{q}{\div\phi}_{0,\om}
-\bscp{\nu\na(2\tilde{v}+\phi)}{\na\phi}_{0,\om}\\
&\qquad-2\kappa(\om,\ga_{D})\norm{\div\phi}_{0,\om}
\Big((\nu_{-}^{-1/2}\nu_{+}^{1/2}+1)c_{fp}(\om,\ga_{D})\norm{\Div\tau+F}_{1,\om}\\
&\qquad\qquad+2\nu_{+}^{1/2}\bnorm{\nu^{-1/2}(\tau+q\,\I-\nu\na\psi)}_{0,\om}
+2\nu_{+}\kappa(\om,\ga_{D})\norm{\div\psi}_{0,\om}\Big).
\end{align*}
In particular, $\psi=\tilde{v}$ is possible.
\end{theo}

For solenoidal $\phi$, i.e., $\phi\in\s_{-1,\ga_{D}}(\om)$ we simply get
$$\norm{\nu^{1/2}\na(v-\tilde{v})}_{0,\om}^2
\geq2\scp{F}{\phi}_{0,\om}
-\bscp{\nu\na(2\tilde{v}+\phi)}{\na\phi}_{0,\om}$$
and equality holds for $\phi=v-\tilde{v}$, provided that
the approximation $\tilde{v}$ is also solenoidal, i.e.,
$\tilde{v}\in v_{D}+\s_{-1,\ga_{D}}(\om)$.
To handle a very non-conforming approximation $\tilde{T}\in\lt(\om)$
we can simply utilize for all $\varphi\in v_{D}+\ho_{-1,\ga_{D}}(\om)$
the triangle inequality
$$\norm{\nu^{1/2}(\na v-\tilde{T})}_{0,\om}
\geq\norm{\nu^{1/2}\na(v-\varphi)}_{0,\om}
-\norm{\nu^{1/2}(\na\varphi-\tilde{T})}_{0,\om}$$
in combination with Theorem \ref{aposttheo6extdom} ($\tilde{v}=\varphi$).
More precisely, we note the following result:

\begin{theo}[a posteriori error estimate VII for exterior domains]
\label{aposttheo7extdom}
Let $\tilde{T}\in\lt(\om)$.
Then for all $\phi\in\ho_{-1,\ga_{D}}(\om)$, all $\tau\in\tilde{\d}_{\ga_{N}}(\om)$,
all $\varphi,\psi\in v_{D}+\ho_{-1,\ga_{D}}(\om)$, and all $q\in\lt(\om)$
\begin{align*}
\norm{\nu^{1/2}(\na v-\tilde{T})}_{0,\om}
&\geq\norm{\nu^{1/2}\na(v-\varphi)}_{0,\om}
-\norm{\nu^{1/2}(\na\varphi-\tilde{T})}_{0,\om},\\
\norm{\nu^{1/2}\na(v-\varphi)}_{0,\om}^2
&\geq2\scp{F}{\phi}_{0,\om}
+2\scp{q}{\div\phi}_{0,\om}
-\bscp{\nu\na(2\varphi+\phi)}{\na\phi}_{0,\om}\\
&\quad-2\kappa(\om,\ga_{D})\norm{\div\phi}_{0,\om}
\Big((\nu_{-}^{-1/2}\nu_{+}^{1/2}+1)c_{fp}(\om,\ga_{D})\norm{\Div\tau+F}_{1,\om}\\
&\qquad+2\nu_{+}^{1/2}\bnorm{\nu^{-1/2}(\tau+q\,\I-\nu\na\psi)}_{0,\om}
+2\nu_{+}\kappa(\om,\ga_{D})\norm{\div\psi}_{0,\om}\Big).
\end{align*}
\end{theo}

\subsection{Applications for 2D Exterior Domains}
\label{2Dsec}

For a Lipschitz domain $\calD\subset\reals^{2}$ we introduce modified polynomially weighted spaces using logarithms by
\begin{align*}
\lt_{\pm1,\ln}(\calD)
&:=\setb{\phi\in\lt_{\mathsf{loc}}(\calD)}{\big(\rho\ln(e+\rho)\big)^{\pm1}\phi\in\lt(\calD)},\qquad
e:\text{Euler's number},\\
\ho_{-1,\ln}(\calD)
&:=\setb{\phi\in\lt_{-1,\ln}(\calD)}{\na\phi\in\lt(\calD)}.
\end{align*}
Note that at infinity $\big(\rho\ln(e+\rho)\big)^{\pm1}$ behaves like $(r\ln r)^{\pm1}$.
The Inner product in $\lt_{\pm1,\ln}(\calD)$ is defined and denoted by
$\scp{\,\cdot\,}{\,\cdot\,}_{\pm1,\ln,\calD}
:=\bscp{\big(\rho\ln(e+\rho)\big)^{\pm2}\,\cdot\,}{\,\cdot\,}_{0,\calD}$. 
All other weighted spaces and norms etc. are modified and defined in the same way.

Let $\om\subset\reals^{2}$ and $\omega\subset\reals^{2}$ be defined as in Section \ref{secprelim}, i.e., 
$\om\subset\reals^{2}$ is an exterior Lipschitz domain.
The situation is now different from the case $d\geq3$ as the constants will be integrable in our weighted spaces.
More precisely, for $0<\epsilon<1$
\begin{align*}
(r\ln r)^{-1}&\in\lt(B_{\epsilon}),
&(r\ln r)^{-1}&\not\in\lt(B_{1+\epsilon}\setminus\overline{B}_{1-\epsilon}),\\
(r\ln r)^{-1}&\in\lt(\reals^{2}\setminus\overline{B}_{e}).
\end{align*}
Introducing
$$\ho_{-1,\ln,\emptyset}(\om)
:=\ho_{-1,\ln}(\om)\cap\reals^{\bot_{-1,\ln,\om}}$$
we have the following Friedrichs/Poincare estimate:

\begin{lem}[Friedrichs/Poincar\'e estimate for 2D exterior domains]
\label{ptheoextdom2D}
There exists $c>0$ such that 
$\norm{v}_{-1,\ln,\om}
\leq c\norm{\na v}_{0,\om}$
for all $v\in\ho_{-1,\ln,\ga_{D}}(\om)$.
The best constant $c$ will be denoted by $c_{fp}(\om,\ga_{D})$.
In the special case $B_{e}\subset\reals^{2}\setminus\om$
and $\ga_{D}=\ga$ it holds $c_{fp}(\om,\ga_{D})\leq2$.
\end{lem}

Note that we need boundary or mean value conditions
as in the case of a bounded domain.

\begin{proof}
From \cite[Appendix 4.2, Lemma 4.1, Corollary 4.2, Remark 4.3]{paulyrepinell},
see also \cite[Lemma 4.1]{saranenwitschexteriorell}, 
we have for all $v\in\ci_{\ga}(\om)$
$$\norm{v}_{-1,\ln,\om}
\leq\norm{(r\ln r)^{-1}v}_{0,\om}
\leq2\norm{\na v}_{0,\om},$$
provided that, e.g., $B_{e}\subset\reals^{2}\setminus\om$,
which extends by density and continuity to $\ho_{-1,\ln,\ga}(\om)$,
i.e., all $v\in\ho_{-1,\ln,\ga}(\om)$,
\begin{align}
\label{fcorextdom2Dapp}
\norm{v}_{-1,\ln,\om}
\leq2\norm{\na v}_{0,\om}.
\end{align}
Let $v\in\ho_{-1,\ln}(\om)$ and 
let us assume w.l.o.g. $r_{1}>e$, $r_{2}:=r_{1}+1$ and $\xi'_{\infty}\leq1$.
Then $\eta v\in\ho_{-1,\ln,\ga}(\supp\eta)$ and 
$\norm{\eta v}_{-1,\ln,\om}
\leq2\norm{\na(\eta v)}_{0,\om}$
by \eqref{fcorextdom2Dapp} for $\om=\supp\eta$. Hence
\begin{align*}
\norm{v}_{-1,\ln,\om}
&\leq2\norm{\na v}_{0,\om}
+2\norm{v\na\eta}_{0,\om}
+\norm{(1-\eta)v}_{-1,\ln,\om}\\
&\leq2\norm{\na v}_{0,\om}
+2\norm{v}_{0,\omega}
+\norm{v}_{0,\omega},
\end{align*}
showing for all $v\in\ho_{-1,\ln}(\om)$
\begin{align}
\label{flemcptpertextdom2Dapp}
\norm{v}_{-1,\ln,\om}
\leq2\norm{\na v}_{0,\om}+3\norm{v}_{0,\omega}.
\end{align}
Now, if the assertion of Lemma \ref{ptheoextdom2D}
is false, there is a sequence $(v_{n})\subset\ho_{-1,\ln,\ga_{D}}(\om)$
with $\norm{v_{n}}_{-1,\ln,\om}=1$ and $\norm{\na v_{n}}_{0,\om}<1/n$. 
Hence, $(v_{n})$ is bounded in $\ho(\omega)$ as well. By Rellich's selection theorem we can assume w.l.o.g.
that $(v_{n})$ already converges in $\lt(\omega)$. 
Thus, by \eqref{flemcptpertextdom2Dapp} $(v_{n})$ is a Cauchy sequence 
in $\lt_{-1,\ln}(\om)$ and hence also in $\ho_{-1,\ln,\ga_{D}}(\om)$. 
Therefore, $(v_{n})$ converges in $\ho_{-1,\ln,\ga_{D}}(\om)$
to some $v\in\ho_{-1,\ln,\ga_{D}}(\om)$ with $\na v=0$. We conclude that $v$ is constant
and hence $v=0$, which implies a contradiction by
$1=\norm{v_{n}}_{-1,\ln,\om}\to0$.
\end{proof}

Now, all results from the sections for $d\geq3$ follow with the obvious modifications,
where we just present the most relevant ones.

\begin{lem}[stability lemma for 2D exterior domains]
\label{stablemextdom2D}
There exists $c>0$ such that for all $f\in\lt(\om)$ 
there is a vector field $v\in\ho_{-1,\ln,\ga_D}(\om)$ with
$$\div v=f\quad\text{and}\quad\norm{\na v}_{0,\om}\leq c\norm{f}_{0,\om}.$$
The best constant is denoted by $\kappa(\om,\ga_{D})$
which equals the norm of the corresponding right inverse $f\mapsto v$. 
Moreover, with $\kappa$ from \eqref{defkappaupperbound}
\begin{align*}
\kappa(\om,\ga_{D})\leq
\hat{\kappa}(\om,\ga_{D})
&:=(1+\kappa)\big(1+c_{fp}(\reals^{2})\frac{\xi'_{\infty}\rho(r_{2})\ln\big(e+\rho(r_{2})\big)}{r_{2}-r_{1}}\big).
\end{align*}
In particular, it holds
$\hat{\kappa}(\om,\ga_{D})\leq(1+\kappa)\big(1+c_{fp}(\reals^{2})\rho(r_{2})\ln(e+\rho(r_{2}))\big)$
for $r_{2}=r_{1}+1$ and $\xi'_{\infty}\leq1$.
If $f$ has compact support in $\overline{\omega}$
and if additionally $\int_{\omega}f=\int_{\om}f=0$ in the case $\ga_{D}=\ga$,
then $v$ can be chosen with compact support in $\overline{\omega}$, 
especially $v\in\ho_{\gamma_D}(\omega)\subset\ho_{\ga_D}(\om)$.
In this case, $\kappa(\om,\ga_{D})\leq\kappa(\omega,\gamma_{D})$.
\end{lem}

\begin{cor}[distance lemma for 2D exterior domains]
\label{distlemextdom2D}
For any $v\in\ho_{-1,\ln,\ga_{D}}(\om)$ there exists 
a solenoidal $v_{0}\in\s_{-1,\ln,\ga_{D}}(\om)$ such that
$$\dist\big(v,\s_{-1,\ln,\ga_{D}}(\om)\big)
=\inf_{\phi\in\s_{-1,\ln,\ga_{D}}(\om)}\bnorm{\na(v-\phi)}_{0,\om}
\leq\bnorm{\na(v-v_{0})}_{0,\om}
\leq\kappa(\om,\ga_{D})\norm{\div v}_{0,\om}.$$
\end{cor}

\begin{cor}[inhomogeneous distance lemma for 2D exterior domains]
\label{inhomodistlemextdom2D}
For any vector field $v\in\ho_{-1,\ln}(\om)$ there exists 
a solenoidal $v_{0}\in\s_{-1,\ln}(\om)$ such that
$v_{0}-v\in\ho_{-1,\ln,\ga_{D}}(\om)$, i.e., $v_{0}|_{\ga_{D}}=v|_{\ga_{D}}$, and
$\bnorm{\na(v_{0}-v)}_{0,\om}
\leq\kappa(\om,\ga_{D})\norm{\div v}_{0,\om}$.
\end{cor}

\begin{cor}[inf-sup lemma for 2D exterior domains]
\label{infsuplemextdom2D}
It holds
$$\inf_{f\in\lt(\om)}\sup_{v\in\ho_{-1,\ln,\ga_{D}}(\om)}
\frac{\scp{f}{\div v}_{0,\om}}{\norm{f}_{0,\om}\norm{\na v}_{0,\om}}\geq\frac{1}{\kappa(\om,\ga_{D})}.$$
\end{cor}

\begin{cor}[Stokes lemma for 2D exterior domains]
\label{stokeslemextdom2D}
For $\nu$, $F\in\lt_{1,\ln,\ga_{N}}(\om)$, and $v_{D}\in\s_{-1,\ln}(\om)$
the 2D Stokes system is uniquely solvable with a solenoidal vector field
$v=v_{D}+\tilde{v}\in v_{D}+\s_{-1,\ln,\ga_{D}}(\om)\subset\s_{-1,\ln}(\om)$
and $p\in\lt(\om)$. Moreover, 
\begin{align*}
\nu\norm{\na\tilde{v}}_{0,\om}
&\leq c_{fp}(\om,\ga_{D})\norm{F}_{1,\ln,\om}
+\nu\norm{\na v_{D}}_{0,\om},\\
\nu\norm{\na v}_{0,\om}
&\leq c_{fp}(\om,\ga_{D})\norm{F}_{1,\ln,\om}
+2\nu\norm{\na v_{D}}_{0,\om},\\
\norm{p}_{0,\om}
&\leq 2\kappa(\om,\ga_{D})\big(c_{fp}(\om,\ga_{D})\norm{F}_{1,\ln,\om}
+\nu\norm{\na v_{D}}_{0,\om}\big).
\end{align*}
\end{cor}

Here we have introduced
\begin{align*}
\lt_{1,\ln,\ga_{N}}(\om)
&:=\begin{cases}\lt_{1,\ln}(\om)&\text{, if }\ga_{D}\neq\emptyset,\\
\lt_{1,\ln,\bot}(\om)&\text{, if }\ga_{D}=\emptyset,\end{cases}\\
\lt_{1,\ln,\bot}(\om)
&:=\lt_{1,\ln}(\om)\cap(\reals^{2})^{\bot_{0,\om}}
=\setb{\phi\in\lt_{1,\ln}(\om)}{\int_{\om}\phi_{i}=0}.
\end{align*}

\subsubsection{A Posteriori Error Estimates for Stationary Stokes Equations in 2D}

We introduce 
$$\tilde{\d}(\om)
:=\setb{\Theta\in\lt(\om)}{\Div\Theta\in\lt_{1,\ln}(\om)}$$
and $\tilde{\d}_{\ga_{N}}(\om)$ as closure of 
$\ci_{\ga_{N}}(\om)$-tensor fields in the norm of $\tilde{\d}(\om)$.

For the approximation of the velocity field we have the following result:

\begin{theo}[a posteriori error estimate I for 2D exterior domains]
\label{aposttheo2extdom2D}
Let the approximation $\tilde{v}$ belong to $v_{D}+\ho_{-1,\ln,\ga_{D}}(\om)$.
Then for all $\tau\in\tilde{\d}_{\ga_{N}}(\om)$ and all $q\in\lt(\om)$
\begin{align*}
\norm{\nu^{1/2}\na(v-\tilde{v})}_{0,\om}
&\leq\nu_{-}^{-1/2}c_{fp}(\om,\ga_{D})\norm{\Div\tau+F}_{1,\ln,\om}
+\bnorm{\nu^{-1/2}(\tau+q\,\I-\nu\na\tilde{v})}_{0,\om}\\
&\qquad+2\nu_{+}^{1/2}\kappa(\om,\ga_{D})\norm{\div\tilde{v}}_{0,\om}.
\end{align*}
\end{theo}

If additionally $\div\tilde{v}=0$ in $\reals^{2}\setminus\overline{B}_{r_{2}}$ and, 
if $\ga_{D}=\ga$, \eqref{pwdecay} holds, then $\kappa(\om,\ga_{D})$
can be replaced by $\kappa(\omega,\gamma_{D})$ in Theorem \ref{aposttheo2extdom2D}.
If the approximation $\tilde{v}$ is solenoidal, i.e., $\div\tilde{v}=0$ in $\om$,
the upper bound coincides with the norm of the error on the left hand side
if $(\tau,q)=(\sigma,p)$.

For the approximation of the pressure function we get:

\begin{theo}[a posteriori error estimate II for 2D exterior domains]
\label{aposttheo4extdom2D}
Let $\tilde{p}\in\lt(\om)$.
Then for all $\tau\in\tilde{\d}_{\ga_{N}}(\om)$
and all $\psi\in v_{D}+\ho_{-1,\ln,\ga_{D}}(\om)$ it holds
\begin{align*}
\norm{p-\tilde{p}}_{0,\om}
&\leq\kappa(\om,\ga_{D})
\Big((\nu_{-}^{-1/2}\nu_{+}^{1/2}+1)c_{fp}(\om,\ga_{D})\norm{\Div\tau+F}_{1,\ln,\om}\\
&\qquad+2\nu_{+}^{1/2}\bnorm{\nu^{-1/2}(\tau+\tilde{p}\,\I-\nu\na\psi)}_{0,\om}
+2\nu_{+}\kappa(\om,\ga_{D})\norm{\div\psi}_{0,\om}\Big).
\end{align*}
\end{theo}

For non-conforming approximations of the velocity field we see:

\begin{theo}[a posteriori error estimate III for 2D exterior domains]
\label{aposttheo5extdom2D}
Let $\tilde{T}\in\lt(\om)$ and $\tilde{p}\in\lt(\om)$.
Then for all $\psi\in v_{D}+\ho_{-1,\ln,\ga_{D}}(\om)$,
all $\tau\in\tilde{\d}_{\ga_{N}}(\om)$, and all $q\in\lt(\om)$ it holds
\begin{align*}
\norm{\nu^{1/2}(T-\tilde{T})}_{0,\om}
&\leq\nu_{-}^{-1/2}c_{fp}(\om,\ga_{D})\norm{\Div\tau+F}_{1,\ln,\om}
+\bnorm{\nu^{-1/2}(\tau+q\,\I-\nu\na\psi)}_{0,\om}\\
&\qquad+2\nu_{+}^{1/2}\kappa(\om,\ga_{D})\norm{\div\psi}_{0,\om}
+\norm{\nu^{1/2}(\na\psi-\tilde{T})}_{0,\om}\\
&\leq\nu_{-}^{-1/2}c_{fp}(\om,\ga_{D})\norm{\Div\tau+F}_{1,\ln,\om}
+\bnorm{\nu^{-1/2}(\tau+q\,\I-\nu\tilde{T})}_{0,\om}\\
&\qquad+2\nu_{+}^{1/2}\kappa(\om,\ga_{D})\norm{\div\psi}_{0,\om}
+2\norm{\nu^{1/2}(\na\psi-\tilde{T})}_{0,\om},\\
\norm{p-\tilde{p}}_{0,\om}
&\leq\kappa(\om,\ga_{D})
\Big((\nu_{-}^{-1/2}\nu_{+}^{1/2}+1)c_{fp}(\om,\ga_{D})\norm{\Div\tau+F}_{1,\ln,\om}\\
&\qquad+2\nu_{+}^{1/2}\bnorm{\nu^{-1/2}(\tau+\tilde{p}\,\I-\nu\tilde{T})}_{0,\om}
+2\nu_{+}\kappa(\om,\ga_{D})\norm{\div\psi}_{0,\om}\\
&\qquad\qquad+2\nu_{+}^{1/2}\norm{\nu^{1/2}(\na\psi-\tilde{T})}_{0,\om}\Big).
\end{align*}
\end{theo}

For $\tilde{T}=\na\tilde{v}$, $\psi=\tilde{v}\in v_{D}+\ho_{-1,\ln,\ga_{D}}(\om)$
we get back Theorem \ref{aposttheo2extdom2D} and Theorem \ref{aposttheo4extdom2D}.
Moreover, using the Helmholtz decomposition 
$$T-\tilde{T}=\na w+\tilde{T}_{0}\in\na\ho_{-1,\ln,\ga_{D}}(\om)\oplus_{0,\nu}\nu^{-1}{}_{0}\d_{\ga_{N}}(\om),$$
we observe
$$\norm{\nu^{1/2}\tilde{T}_{0}}_{0,\om}
=\min_{\psi\in v_{D}+\ho_{-1,\ln,\ga_{D}}(\om)}\bnorm{\nu^{1/2}(\na\psi-\tilde{T})}_{0,\om}.$$

As before, error estimates for the stress tensor field $\sigma$ 
follow immediately by the triangle inequality. For a lower bound we have the following result:

\begin{theo}[a posteriori error estimate IV for 2D exterior domains]
\label{aposttheo6extdom2D}
Let the approximation $\tilde{v}$ belong to $v_{D}+\ho_{-1,\ln,\ga_{D}}(\om)$.
Then for all $\phi\in\ho_{-1,\ln,\ga_{D}}(\om)$, all $\tau\in\tilde{\d}_{\ga_{N}}(\om)$,
all $\psi\in v_{D}+\ho_{-1,\ln,\ga_{D}}(\om)$,
and all $q\in\lt(\om)$ it holds
\begin{align*}
&\qquad\norm{\nu^{1/2}\na(v-\tilde{v})}_{0,\om}^2\\
&\geq2\scp{F}{\phi}_{0,\om}
+2\scp{q}{\div\phi}_{0,\om}
-\bscp{\nu\na(2\tilde{v}+\phi)}{\na\phi}_{0,\om}\\
&\qquad-2\kappa(\om,\ga_{D})\norm{\div\phi}_{0,\om}
\Big((\nu_{-}^{-1/2}\nu_{+}^{1/2}+1)c_{fp}(\om,\ga_{D})\norm{\Div\tau+F}_{1,\ln,\om}\\
&\qquad\qquad+2\nu_{+}^{1/2}\bnorm{\nu^{-1/2}(\tau+q\,\I-\nu\na\psi)}_{0,\om}
+2\nu_{+}\kappa(\om,\ga_{D})\norm{\div\psi}_{0,\om}\Big).
\end{align*}
In particular, $\psi=\tilde{v}$ is possible.
\end{theo}

Again, for solenoidal $\phi$, i.e., $\phi\in\s_{-1,\ln,\ga_{D}}(\om)$ we simply get
$$\norm{\nu^{1/2}\na(v-\tilde{v})}_{0,\om}^2
\geq2\scp{F}{\phi}_{0,\om}
-\bscp{\nu\na(2\tilde{v}+\phi)}{\na\phi}_{0,\om}$$
and equality holds for $\phi=v-\tilde{v}$, provided that
the approximation $\tilde{v}$ is also solenoidal, i.e.,
$\tilde{v}\in v_{D}+\s_{-1,\ln,\ga_{D}}(\om)$.
Finally, to handle also very non-conforming approximations $\tilde{T}\in\lt(\om)$
we can simply utilize for all $\varphi\in v_{D}+\ho_{-1,\ln,\ga_{D}}(\om)$
the triangle inequality
$$\norm{\nu^{1/2}(\na v-\tilde{T})}_{0,\om}
\geq\norm{\nu^{1/2}\na(v-\varphi)}_{0,\om}
-\norm{\nu^{1/2}(\na\varphi-\tilde{T})}_{0,\om}$$
in combination with Theorem \ref{aposttheo6extdom2D} ($\tilde{v}=\varphi$).

\bibliographystyle{plain}

\end{document}